\documentclass[a4paper,fleqn]{cas-sc}

\usepackage[numbers]{natbib}
\usepackage{xcolor}
\usepackage{pdflscape}
\usepackage{subfigure}
\usepackage{amsmath,amssymb,amsfonts,amsthm}
\usepackage{graphicx}
\usepackage{geometry}
\usepackage{float}
\usepackage{xcolor}
\usepackage{hyperref}
 \hypersetup{colorlinks,citecolor=blue,linkcolor=blue,breaklinks=true}
\usepackage{booktabs}
\usepackage{colortbl}
\usepackage{enumitem}
\usepackage{epstopdf}
\usepackage{algorithm}
\usepackage{algpseudocode}
\usepackage{array}
\usepackage{bbding}
\usepackage{fancyhdr}
\usepackage{fancyvrb} %
\usepackage{longtable}
\usepackage{listings}
\usepackage{rotating,rotfloat}
\usepackage{setspace}

\usepackage{msc}
\usepackage[font=small,labelfont=bf,labelsep=none]{caption}

\newtheorem{theorem}{Theorem}[section]
\newtheorem{lemma}{Lemma}[section]

\newtheorem{proposition}[theorem]{Proposition}

\newtheorem{remark}{Remark}[section]
\numberwithin{equation}{section} 

\begin{document}
\begin{sloppypar}

\let\WriteBookmarks\relax
\def\floatpagepagefraction{1}
\def\textpagefraction{.001}
\shorttitle{}
\shortauthors{Zhao et al.}

\title [mode = title]{Random Batch Method with Momentum Correction}

\author[1]{Yanshun Zhao}[style=chinese]
\ead{zhaoyanshun905@gmail.com}

\author[2]{Jingrun Chen}[style=chinese]
\cormark[1]
\author[3]{Zhiwen Zhang}[style=chinese]
\cormark[1]

\cortext[cor1]{Corresponding author}

\address[1]{School of Mathematical Sciences, University of Science and Technology of China, Hefei, 230026, People's Republic of China}
\address[2]{School of Mathematical Sciences, University of Hong Kong, Pokfulam, Hong Kong}
\address[3]{School of Mathematical Sciences and Suzhou Institute for Advanced Research, University of Science and Technology of China, Suzhou 215132, People's Republic of China}

\begin{abstract}
The Random Batch Method (RBM) is an effective technique to reduce the computational complexity when solving certain stochastic differential problems (SDEs) involving interacting particles. It can transform the computational complexity from O($N^2$) to O($N$), where N represents the number of particles. However, the traditional RBM can only be effectively applied to interacting particle systems with relatively smooth kernel functions to achieve satisfactory results. To address the issue of non-convergence of the RBM in particle interaction systems with significant singularities, we propose some enhanced methods to make the modified algorithm more applicable. The idea for improvement primarily revolves around a momentum-like correction, and we refer to the enhanced algorithm as the Random Batch Method with Momentum Correction (  RBM-M). We provide a theoretical proof to control the error of the algorithm, which ensures that under ideal conditions it has a smaller error than the original algorithm. Finally, numerical experiments have demonstrated the effectiveness of the   RBM-M algorithm.

\end{abstract}



\maketitle
\section{Introduction}
\subsection{Interacting particle systems}
Interacting particle systems are a class of stochastic models that have gained significant attention in the realm of applied mathematics and theoretical physics. These systems, comprised of a large number of individual particles or agents, exhibit complex and often emergent behaviors that arise from the intricate interactions among the particles\cite{10.1093/biomet/60.3.581,Swart2017ACI}. In general, the interacting particle systems can be represented by the following equation:
\begin{align} \label{ipm}
    d X^{i}=-\nabla V\left(X^{i}\right) d t+\frac{1}{N-1} \sum_{j \neq i} K\left(X^{i}-X^{j}\right) d t+\sigma d B^{i}.
\end{align}
In this stochastic differential equation, $K(\cdot)$ is the kernel of interacting particle system, $V$  represents the external flow independent of the interacting particle systems, $X^{i}$ represents some property of the particle, such as its position, and $B$ is a stochastic disturbance while the diffusion term $\sigma$ is represents the strength of the disturbance.

At the core of interacting particle systems lies the concept of randomness and probability. Each particle or agent within the system is governed by a set of stochastic differential equations, which capture the dynamic evolution of the system over time. These equations incorporate the interactions between the particles, enabling the modeling of a wide range of phenomena, from the diffusion of molecules in a fluid to the collective behavior of social networks.

The mathematical analysis of interacting particle systems requires an in-depth understanding of probability theory, stochastic processes, and partial differential equations. In addition to their theoretical significance, interacting particle systems are widely used in many fields, such as physics, chemistry, biology, economics, and social sciences \cite{Degond2015CoagulationFragmentationMF,Vicsek1995NovelTO,Carrillo2012UniquenessFK,Motsch2013HeterophiliousDE}. From modeling the spread of outbreaks to understanding financial market dynamics, these systems have proven to be powerful frameworks for capturing the complex interactions in the world around us.

Due to the significant theoretical and practical implications of interacting particle systems, finding effective solutions to these systems is of great importance. In other words, devising accurate and efficient numerical methods for solving interacting particle systems is a problem worthy of in-depth investigation.

Upon initial consideration, the solution to this problem may appear to be a straightforward numerical approach to solving the underlying stochastic differential equations, which can be addressed using techniques such as the Euler method \cite{FerreiroCastilla2013AnES} or Runge-Kutta method \cite{Hong2017StrongCR}. However, further examination reveals certain challenges. Examining equation (\ref{ipm}), we can observe that if we attempt to solve it directly using numerical methods, the computational complexity will scale quadratically with the number of particles, $N$, in the interacting particle system. This is because for each particle, the equation accounts for the interactions with all other particles, resulting in a complexity of $O(N)$. Given the N particles in the system, the overall computational complexity becomes $O(N^2)$.

For interacting particle systems with a large number of particles, this quadratic computational complexity poses a significant challenge, as it leads to lengthy simulation times. Therefore, if we aim to solve large-scale interacting particle systems in industrial applications, it may not be appropriate to directly apply methods like the Euler or Runge-Kutta techniques to obtain numerical solutions to the stochastic differential equations. Instead, we need to develop improved algorithms that can reduce the computational complexity as much as possible while maintaining the necessary accuracy.

\subsection{Related work}
To mitigate the computational complexity inherent in simulating interacting particle systems, various methods have been proposed by previous researchers. For instance, the Fast Multipole Method (FMM) has been a widely adopted technique since the late 20th century \cite{GREENGARD1987325}. The FMM provides a means to reduce the computational complexity to $O(NlogN)$, or even $O(N)$ in certain cases, by exploiting the inherent hierarchical structure of the problem. The core idea of the FMM is to partition the computational domain into a hierarchical tree-like structure, where each node in the tree represents a cluster of particles or objects. Through the use of a series of multipole expansions and local Taylor series expansions, the algorithm is able to efficiently approximate the long-range interactions between distant clusters, while accurately computing the short-range interactions between nearby particles. The implementation of the FMM involves several key steps, including the construction of the hierarchical tree structure, the computation of multipole and local Taylor series expansions, and the efficient traversal of the tree to accumulate the contributions from distant clusters. The algorithm also employs various optimization techniques, such as adaptive refinement of the tree structure and the utilization of parallel computing architectures, to further enhance its computational performance.

Additionally, the Tree Code method is another fast algorithm for calculating multi-body interactions. This method also utilizes a hierarchical tree structure to represent the particle distribution, and can significantly reduce the computational complexity by aggregating the contributions from distant particles \cite{Bagla1999TreePMAC}.

Furthermore, there exist mesh-based methods, such as the Particle-Mesh Method \cite{YAO2021239}, which transform the continuous space into a grid and perform calculations on the grid. This approach can leverage fast Fourier transform and other techniques to efficiently compute long-range interactions, thereby substantially reducing the computational complexity.

In recent years, a novel algorithm known as the Random Batch Method (RBM) has gained significant attention from many researchers due to its simplicity of implementation and impressive computational efficiency \cite{JIN2020108877,Jin2021RandomBM}. The core principle underlying RBM is to divide the particle system into smaller, randomly selected batches, and then perform the computations on these batches rather than the entire system. This approach effectively mitigates the computational complexity from the traditional $O(N^2)$ scaling to a more manageable $O(N)$ or even better scaling. The key advantage of RBM lies in its ability to approximate the long-range interactions between distant particles by considering only the nearby particles within each randomly selected batch. Furthermore, RBM can be easily parallelized, and this parallelization capability serves to further enhance the computational performance of the algorithm, making it well-suited for simulating large-scale particle systems. Notably, the implementation of RBM is much simpler compared to the more sophisticated Fast Multipole Method (FMM) and tree code methods. This unique combination of simplicity and high computational efficiency has rendered RBM a popular choice among researchers and practitioners across a wide range of applications, including fluid mechanics, plasma physics, and molecular dynamics.

\subsection{Advantage of Algorithm}
However, despite the good results achieved by the above fast algorithms, there are still some small limitations. Therefore, this paper proposes the RBM with Momentum Correction (  RBM-M) algorithm, which is an evolution of the standard RBM algorithm and offers several advantages. The main contributions of this work are as follows:
\begin{enumerate}
\item The   RBM-M algorithm shares a similar underlying concept with the RBM algorithm, making it relatively simple to understand and implement.
\item In the face of particle systems with relatively large singularities, the   RBM-M algorithm can achieve higher accuracy compared to the standard RBM (see theorem \ref{rbmm}), while maintaining a similar computational time. Numerical experiments demonstrate that the error of   RBM-M can be reduced to half that of RBM in some cases (Detail in table \ref{table_3}).
\item The paper provides a theoretical analysis, proving an upper bound on the error of the   RBM-M algorithm. It also analyzes in detail the reasons why the error is smaller than that of the RBM algorithm. In the other hand, from a probabilistic perspective, the statistical properties of  RBM-M are demonstrated (Detail in the theorem \ref{the_mean}, the theorem \ref{the_var} and the figure \ref{figure_7}).
\item The   RBM-M algorithm introduces a hyperparameter $\beta$, which allows for the flexible control of the correction strength. As the singularity of the system increases, $\beta$ can be adjusted to obtain higher accuracy. Consequently, the   RBM-M algorithm exhibits improved robustness and flexibility, making it suitable for a wider range of example systems.
\end{enumerate}

In Section \ref{section2}, we will be devoted to our proposed algorithm, a momentum-based modification of the random batch method, henceforth referred to as   RBM-M. Then we will provide a rigorous analysis of the error characteristics of this   RBM-M approach in Section \ref{section3}. Section \ref{section4} will then detail numerical experiments conducted to validate the efficacy of the conclusions and algorithms associated with our   RBM-M method. Finally, Section \ref{section5} will discuss potential future research directions and offer concluding remarks.

\section{RBM-M}\label{section2}
\subsection{The deficiency of RBM}
While the Random Batch Method (RBM) offers numerous advantages, it also possesses a certain limitation. Specifically, in order to obtain more accurate results, it is necessary to impose certain restrictions on the interacting particle system to ensure that the system's singularity is not excessively large. The following conclusions provide an analysis of the error associated with RBM under specific conditions \cite{JIN2020108877}.

\begin{proposition}\label{rbm}
Suppose \( V \) is strongly convex on \( \mathbb{R}^d \) so that \( x \mapsto V(x) - \frac{1}{2}|x|^2 \) is convex, and \( \nabla V \), \( \nabla^2 V \) have polynomial growth (i.e., \( |\nabla V(x)| + |\nabla^2 V(x)| \leq C(1 + |x|^q) \) for some \( q > 0 \)). Assume \( K(\cdot) \) is bounded, Lipschitz on \( \mathbb{R}^d \) with Lipschitz constant \( L \) and has bounded second order derivatives. Then the error of RBM:

\begin{align*}
    \sup_{t\ge 0} \parallel Z^1(t) \parallel \le
    C\sqrt{\frac{\tau}{p-1}+\tau^2}\le C\sqrt{\tau}.
\end{align*}
\end{proposition}
It can be observed that in order to minimize the numerical error of the RBM as the time discretization step size $\tau$ approaches zero, it is necessary to ensure that the kernel function $K$ of the interacting particle system is relatively smooth, and the external flow $V$ also satisfies certain properties.

However, not all interacting particle systems possess these desirable properties. Furthermore, when the singularity of the particle system is relatively large, the error associated with RBM can become substantial, and it may even be the case that the method does not converge at all. Therefore, the challenge of improving the algorithm to enhance the stability of RBM and obtain accurate results in the presence of particle systems with relatively large singularities is a problem that merits further investigation.

\subsection{Regularization of kernel}
When the RBM was first introduced, certain restrictions were imposed on the kernel function of the interacting particle systems in order to ensure the numerical solution converges to the true solution. Since these restrictions on the kernel function limit the scope of applicability, it would be desirable to explore ways to relax these constraints, thereby expanding the range of problems that RBM can effectively address\cite{Wang2021DeepParticleLI,Wang2022ADM}.

The way is:
\begin{align}
    K_{\delta}(z)=K(z)\frac{\mid z\mid ^2}{\mid z\mid ^2+{\delta}^2}.
\end{align}

This is a very simple modification, but it can really achieve good results. After this regulation, some kernel are polished so that it satisfies the convergent condition, like Keller-Segal kernel\cite{Liu2016PositivitypreservingAA} and Biot–Savart kernel\cite{Eisterer2004TheSO}. Taking the KS kernel as an example.

For the KS algorithm, the kernel function $K$ in two-dimensional space is:
\[
K(z) =C \cdot\frac{  z}{\|z\|^2}.
\]
Specifically, taking a 2D example:\\
For the x-direction:
\[K_1(b-a) = \frac{C \cdot (b_x-a_x)}{\|b-a\|^2}.\]
For the y-direction:
\[K_2(b-a) = \frac{C \cdot (b_y-a_y)}{\|b-a\|^2}.\]
Where $a$ and $b$ represent the positions of the two examples respectively, and their corresponding angle symbols indicate whether the component is in the x- direction or the y-direction, and $C$ is a constant.

For this kernel function, there is a singularity at z = 0. Hence, a regularization method is adopted:
\[
K_{\delta}(z) = K(z) \cdot \frac{\|z\|^2}{\|z\|^2 + \delta^2} = C \cdot \frac{z}{\|z\|^2 + \delta^2}.
\]
Therefore, the new kernel function $K_{\delta}(z)$ is continuous on $\mathbb{R}^d$, and obviously $K_{\delta}(z)$ is globally stable, hence $K_{\delta}(z)$ is Lipschitz continuous on $\mathbb{R}^d$. Then the interacting particle systems satisfies the requirements of the proposition \ref{rbm}, so its error can be guaranteed by the proposition.

\subsection{Correct algorithm with momentum}
Based on the introduction of previous related work, we have gained a certain understanding of the RBM. The key idea behind RBM is to leverage the concept of randomness to process all particles in batches. Guided by this principle, small and randomly selected batches of particle interactions are utilized, thereby reducing the computational cost of the n-particle system for binary interactions from $O(N^2)$ per time step to $O(N)$ per time step.

Carefully examining the algorithmic steps of RBM, one can observe that its underlying concept is highly similar to that of Stochastic Gradient Descent (SGD) \cite{Ruder2016AnOO,Chen2020UnderstandingGC}. Both methods involve processing the entire dataset by randomly selecting a portion of the samples, with the role of this partial sample approximating that of the entire dataset. Using this random subset instead of the full dataset can significantly reduce the computational complexity.

However, there are certain issues associated with the use of SGD. One of the primary shortcomings is the inherent noise in the gradient estimates. Since the gradient is computed using a single data point or a small batch of data points, the resulting gradient can be noisy and may not accurately represent the true gradient of the entire dataset. This can lead to oscillations in the parameter updates and slow down the convergence of the algorithm, particularly in the later stages of the optimization process. To address this issue, various techniques have been proposed, such as the use of momentum. Momentum-based methods, like the Nesterov Accelerated Gradient (NAG) algorithm, incorporate the previous gradient updates to smooth out the parameter updates and accelerate the convergence \cite{Lin2019NesterovAG,Kong2024QuantitativeCO}.

Upon a careful examination of the challenges encountered with the RBM, it has been observed that the algorithm's accuracy is diminished when applied to interacting particle systems with relatively large singularities. We postulate that this issue arises from a similar shortcoming as that of Stochastic Gradient Descent (SGD) algorithms. Specifically, since the particle interactions are computed using small batches of data points, the resulting interactions may not accurately represent the true interactions of the entire particle system. This can lead to inaccurate calculation of the interaction terms when the governing equations are updated, consequently introducing substantial errors.

Drawing inspiration from the advancements in momentum correction for SGD, we have explored the feasibility of incorporating momentum into the RBM algorithm to rectify its numerical performance. Accordingly, we propose a variant, named RBM with Momentum Correction (  RBM-M), which aims to address the limitations observed in the original RBM formulation:

\begin{algorithm}[H]
\emph{$K\left(X^{i}-X^{j}\right)=0$ when time=0}

\emph{for  m  in  1:[T / $\tau$]  do}

\emph{\ \qquad Divide  $\{1,2, \ldots, p n\}$  into  n  batches randomly.}

\emph{\ \qquad for each batch  $\mathcal{C}_{q}$  do}

\emph{\ \qquad \ \qquad Update  $X^{i}\ 's   \left(i \in \mathcal{C}_{q}\right) $ by solving the following SDE with  $t \in\left[t_{m-1}, t_{m}\right)$ }

\emph{\ \qquad \ \qquad$K_{cor}\left(X^{i}-X^{j}\right)=
\beta K_{t-1}\left(X^{i}-X^{j}\right)+
(1-\beta)K_{t}\left(X^{i}-X^{j}\right)$}

\emph{\ \qquad \ \qquad $d X^{i}=-\nabla V\left(X^{i}\right) d t+\frac{1}{p-1} \sum_{j \in \mathcal{C}_{q}, j \neq i} K_{cor}\left(X^{i}-X^{j}\right) d t+\sigma d B^{i} $}

\emph{\ \qquad end for}

\emph{end for}

\caption{  RBM-M}
\end{algorithm}
The   RBM-M is designed to address the challenges associated with interacting particle systems, where the evolution time spans the interval $[0, T]$, and the discrete time step is denoted as $\tau$. The algorithmic flow can be outlined as follows: For the solution of the governing equations at each time step, the total $N$ particles are divided into $n$ groups, each containing $p$ particles (if the division is not evenly distributed, the final group may have fewer than $p$ particles). For each element within the group $C_p$, the local interactions are computed and used as a proxy for the global interactions to solve the governing equations. After solving the equations for one time step, the particle grouping is reorganized, and the aforementioned process is repeated until the final time $T$ is reached. The interactions calculated within this algorithmic framework serve as the cornerstone for the RBM's approach to modeling interacting particle systems.
\begin{align}
    K_{cor}\left(X^{i}-X^{j}\right)=
\beta K_{t-1}\left(X^{i}-X^{j}\right)+
(1-\beta)K_{t}\left(X^{i}-X^{j}\right).
\end{align}
The   RBM-M employs an exponential weighted averaging technique, where each update step comprises two components: the interaction of the current time instance and the particle dynamics from the previous time step. The hyperparameter $\beta$, which can be determined empirically, plays a crucial role in this process.

When $\beta$ is smaller, the correction effect is less pronounced, and the update equation primarily considers the information from the current time step. Conversely, as $\beta$ increases, the impact of momentum becomes more significant, leading to a stronger corrective effect. However, this also introduces the potential for additional errors.

By performing the k-th evolution, the $K_{cor}$ term incorporates information from all the preceding k evolutions, thereby enhancing the stability and robustness of the algorithm. This approach can be visualized more intuitively through the figure \ref{figure_1} provided.
\begin{figure}[!ht]
        \centering 
        \includegraphics[scale=0.5]{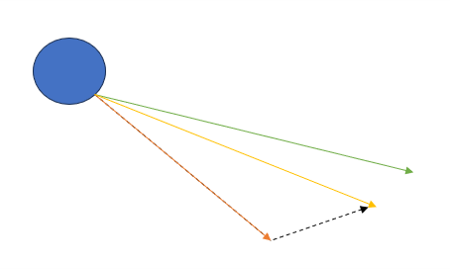}
        \caption{  RBM-M algorithm diagram.}
\label{figure_1}
    \end{figure}
In the diagram, from top to bottom, the first green line represents the real interaction, the second yellow line represents the   RBM-M calculation, and the third orange line represents the RBM calculation. The gray dashed line represents the momentum correction effect, that is, the improvement of RBM by   RBM-M.

\begin{remark}
In the   RBM-M, the hyperparameter $\beta$ exhibits a distinct role compared to Stochastic Gradient Descent (SGD) with momentum. In   RBM-M, the particle position changes at each step, leading to a different distribution at every time step. Consequently, the previous information can introduce new errors in addition to the correction. To mitigate the impact of these introduced errors, the hyperparameter $\beta$ in   RBM-M is required to be relatively small, typically less than 0.1, as further analyzed in the theoretical section. In contrast, SGD with momentum operates on a consistent data distribution, where the gradient of each update is theoretically subject to the same distribution. Therefore, to enhance the algorithm's stability, the updates in SGD with momentum are made in small increments, i.e., $\beta$ is set to a large value, and the newly added information is only modified by a small amplitude.
\end{remark}

\begin{remark}
Compared to the original RBM,   RBM-M introduces additional computations when calculating the interaction. However, this extra information is already available, and no additional data is required. Equivalently, the extra computation per evolution in   RBM-M amounts to two multiplications and one addition, which can be considered negligible. As a result, the actual computational complexity of   RBM-M remains $O(N)$, similar to RBM.
\end{remark}

\section{Theoretical analysis of   RBM-M}\label{section3}
The   RBM-M algorithm has been described in greater detail. Now, we will provide a proof of the algorithm's performance.

\subsection{The error of   RBM-M}
First, we will denote the true solution as $E$. Additionally, we will represent the solution obtained by the RBM algorithm as $u^i_p$ and the solution obtained by the   RBM-M algorithm as $\tilde{u}^i_p$, where the subscript i denotes the i-th iteration, and the subscript p represents the number of particles per group.

Second, let us review the algorithmic process of   RBM-M:
\begin{align}
   \tilde{u}^0_p=u^0_p,
\end{align}
\begin{align}
    \tilde{u}^i_p=\beta \tilde{u}^{i-1}_p+(1-\beta) u^i_p   (i\ge1).
\end{align}

And through some calculations related to exponential weighted averaging we can know that:
\begin{align*}
\tilde{u}^n_p &=\beta \tilde{u}^{n-1}_p+(1-\beta) u^n_p \\
& =\beta (\beta \tilde{u}^{n-2}_p+(1-\beta) u^{n-1}_p)+(1-\beta) u^n_p \\
& =\beta^n u^0_p+\beta^{n-1} (1-\beta) u^1_p+\cdots+
\beta (1-\beta) u^{n-1}_p+(1-\beta) u^n_p\\
&=\beta^n u^0_p+(1-\beta)   \sum_{i=1}^{n} \beta^{n-i}  u^i_p.
\end{align*} 

Next, we will estimate the error of the   RBM-M algorithm. Before that, we should note that based on our previous assumptions, the parameters $\beta$ and $\tau$ are both very small. Then, we will present some propositions that will be used to prove the error bound.

First, we will analyze the errors in adjacent time steps of the equation:
\begin{proposition}\label{proposition1}
Assuming the conditions proposition \ref{rbm} in are met,then
\begin{align}
    \parallel u^t_p-u^{i-1}_p \parallel \le C\tau,\ \ \ \ t\in (t_{i-1},t_i).
\end{align}
\end{proposition}
The proposition can be proved directly by the boundedness of kernel function.

With the preliminary preparation, the error bounds of   RBM-M can be proved:
\begin{theorem}\label{rbmm}
Suppose \( V \) is strongly convex on \( \mathbb{R}^d \) so that \( x \mapsto V(x) - \frac{1}{2}|x|^2 \) is convex, and \( \nabla V \), \( \nabla^2 V \) have polynomial growth (i.e., \( |\nabla V(x)| + |\nabla^2 V(x)| \leq C(1 + |x|^q) \) for some \( q > 0 \)). Assume \( K(\cdot) \) is bounded, Lipschitz on \( \mathbb{R}^d \) with Lipschitz constant \( L \) and has bounded second order derivatives. Then the error of   RBM-M:
\begin{align} \label{error}
    \parallel E-\tilde{u}^{n}_p \parallel \le C(1-\beta^2)\sqrt{\frac{\tau}{p-1}+\tau^2}+C\beta (1-\beta)\tau+O(\beta^2) \lesssim
    C(1-\beta^2)\sqrt{\tau}.
\end{align}
\end{theorem}
\begin{proof}
Substitute the definition of $\tilde{u}^{n}_p$:
\begin{align*}
    \parallel E-\tilde{u}^{n}_p \parallel \le 
    \parallel E-(\beta^n u^0_p+(1-\beta)   \sum_{i=1}^{n} \beta^{n-i}  u^i_p) \parallel.
\end{align*}\\
$\beta$ is vary small so that the high-order terms of $\beta$ can be almost ignored:
\begin{align*}
    &\parallel E-(\beta^n u^0_p+(1-\beta)   \sum_{i=1}^{n} \beta^{n-i}  u^i_p) \parallel \\
    \le &\parallel (1-\beta^2)E-(1-\beta)\beta u^{n-1}_{p1}-(1-\beta) u^n_{p2} \parallel +O(\beta^2),
\end{align*}
where p1 and p2 are two different groups in RBM.\\
Let $u^{n-1}_{p1}$ continue running time $\tau$,considering $u^{n}_{p1}$:
\begin{align*}
    &\parallel (1-\beta^2)E-(1-\beta)\beta u^{n-1}_{p1}-(1-\beta) u^n_{p2} \parallel +O(\beta^2)  \\
    =&\parallel (1-\beta^2)E-(1-\beta)\beta u^{n}_{p1}-(1-\beta) u^n_{p2} +
    (1-\beta)\beta (u^{n}_{p1}-u^{n-1}_{p1})\parallel +O(\beta^2)\\
    \le&  \beta(1-\beta)\parallel E-u^{n}_{p1}\parallel +
    (1-\beta) \parallel E-u^{n}_{p2}\parallel+\beta (1-\beta) \parallel u^{n}_{p1}-u^{n-1}_{p1} \parallel+O(\beta^2).
\end{align*}
By proposition \ref{rbm}:
\begin{align*}
    \parallel E-u^{n}_{p1}\parallel \le C\sqrt{\frac{\tau}{p-1}+\tau^2},
\end{align*}

\begin{align*}
    \parallel E-u^{n}_{p2}\parallel \le C\sqrt{\frac{\tau}{p-1}+\tau^2}.
\end{align*}
By proposition \ref{proposition1}:
\begin{align*}
    \parallel u^{n}_{p1}-u^{n-1}_{p1} \parallel \le C\tau.
\end{align*}
Note that compare to $\tau$, $\tau^2$ is very small, so: 
\begin{align*}
    &\parallel E-\tilde{u}^{n}_p \parallel \\
    \le & C(1-\beta)\sqrt{\frac{\tau}{p-1}+\tau^2}+C\beta(1-\beta)\sqrt{\frac{\tau}{p-1}+\tau^2}+C\beta (1-\beta)\tau+O(\beta^2)\\
    \le &C(1-\beta^2)\sqrt{\frac{\tau}{p-1}+\tau^2}+C\beta (1-\beta)\tau+O(\beta^2)\\
    \lesssim & C(1-\beta^2)\sqrt{\tau}.
\end{align*}
\end{proof}

We have mathematically proven that the error of the   RBM-M \textbf{is smaller than} that of the RBM.

The error of the   RBM-M can be decomposed into two terms. The first term of error (\ref{error}) is:
\begin{align*}
    (1-\beta^2)\sqrt{\frac{\tau}{p-1}+\tau^2},
\end{align*}
which represents the main part of the error and corresponds to the error of the RBM. The second term of error (\ref{error}) is:
\begin{align*}
    \beta (1-\beta)\tau+O(\beta^2),
\end{align*}
which represents the additional error introduced by the momentum term.

Based on the form of the error expression, we hypothesize that the   RBM-M algorithm exhibits better performance when the kernel of the RBM has a relatively large singularity. The rationale is that when the kernel has a large singularity, the error of the RBM ($\sqrt{\frac{\tau}{p-1}+\tau^2}$) is significant. In this case, the effect caused by the coefficient ($1-\beta^2$) becomes more prominent. However, the additional term ($\beta (1-\beta)\tau+O(\beta^2)$) has little relevance to the singularity of the kernel. As long as $\beta$ becomes larger, this error term also becomes larger. Therefore, the overall error of the   RBM-M is a balanced result. When the singularity of the kernel is large, the advantages brought by the momentum correction outweigh the disadvantages caused by the additional error term, enabling the   RBM-M to achieve a more accurate solution. We plan to validate this hypothesis through numerical experiments in the subsequent analysis.

\subsection{Good statistical properties about variance and expectations}
On the other hand, we can explain the advantages of the   RBM-M algorithm from a probabilistic perspective. We first consider the situation of the RBM, for which the following theoretical analysis \cite{JIN2020108877} has been provided, with its expectation and variance respectively:
\begin{proposition}
Consider $p \ge 2$ and a given fixed $x = (x^{1}, \ldots, x^{N}) \in \mathbb{R}^d$. Denote:
\begin{align}\label{e1}
    X_{m,i}(x) = \frac{1}{p - 1} \sum_{j \in C_i, j \neq i} K(x^{i} - x^{j})-\frac{1}{N - 1} \sum_{j \neq i} K(x^{i} - x^{j}),
\end{align}
Then, for all $i$,
\begin{align}
    \mathbb{E}X_{m,i}(x) = 0,
\end{align}
where the expectation is taken with respect to the random division of batches. Moreover, the variance is given by
\begin{align}
    \text{Var}(X_{m,i}(x)) = \mathbb{E}[X_{m,i}(x)]^2 = \left( \frac{1}{p-1} - \frac{1}{N-1} \right) \Lambda_{i}(t),
\end{align}
where
\begin{align}
    \Lambda_{i}(t) = \frac{1}{N - 2} \sum_{j\neq i} \mid K(x^{i} - x^{j}) - \frac{1}{N - 1} \sum_{k\neq i}  K(x^{i} - x^{l})\mid^2.
\end{align}
\end{proposition}
 
Then we consider   RBM-M, denoting:
\begin{align}
    \tilde{K}(x^i_0-x^j_0)=K(x^i_0-x^j_0),
\end{align}
\begin{align}
    \tilde{K}(x^i_n-x^j_n)=
\beta \tilde{K}(x^i_{n-1}-x^j_{n-1})+(1-\beta) K(x^i_n-x^j_n), \ \ \ (n\ge1).
\end{align}
while $x^i , x^j$ means different particles,subscript '$n$' means $n$-th time.\\
Let $\tilde{K}$ is the exponential weighted averaging of interacting particles term:
\begin{align*}
\tilde{K}^i_n=&\beta \tilde{K}_{n-1}+\frac{1-\beta}{p-1} \sum_{j\in C_n,j\neq i}{K}(x^i_{n}-x^j_{n})\\
=&\frac{\beta^n}{p-1} \sum_{j\in C_0,j\neq i}{K}(x^i_{0}-x^j_{0})+\frac{1-\beta}{p-1}   \sum_{s=1}^{n} \beta^{n-s}  \sum_{j\in C_s,j\neq i}{K}(x^i_{s}-x^j_{s}).
\end{align*}
while $C_s$ means different groups.\\
Define:
\begin{align}\label{e2}
    Y_{n,i}(x)=\tilde{K}^i_n-\frac{1}{N-1} \sum_{i \neq j}K(x^i_{n}-x^j_{n}).
\end{align}

Next, we calculate and prove some statistical properties of   RBM-M. The first is unbiaseness, which has the following theorem:
\begin{theorem}\label{the_mean}
The interacting particles term of   RBM-M is asymptotic unbiased estimation of true interacting particles term, which means that:
\begin{align}\label{mean}
    \lim\limits_{\tau \rightarrow 0 \atop \beta \rightarrow 0 } \mathbb{E}(Y_{n,i}(x))=0.
\end{align}
\end{theorem}
\begin{proof}
By the mathematical definition of expectation, we know that
    \[ \mathbb{E}(X_{n,i}(x))=\mathbb{E}(\frac{1}{p-1}\sum_{j\in C,j\neq i}{K}(x^i-x^j)-
    \frac{1}{N-1}\sum_{j\neq i}{K}(x^i-x^j))=0.\]
Then we have:
\begin{align}\label{the2-1}
    \parallel \mathbb{E}(Y_{n,i}(x)) \parallel &\le \beta^n \parallel \mathbb{E}
    (\frac{1}{p-1}\sum_{j\in C_0,j\neq i}{K}(x^i_{0}-x^j_{0})-
    \frac{1}{N-1}\sum_{j\neq i}{K}(x^i_{n}-x^j_{n}))\parallel\\
    &+\cdots+\beta(1-\beta) \parallel \mathbb{E}(\frac{1}{p-1}\sum_{j\in C_{n-1},j\neq i}{K}(x^i_{n-1}-x^j_{n-1})-\frac{1}{N-1}\sum_{j\neq i}{K}(x^i_{n}-x^j_{n}))\parallel \nonumber \\
    &+(1-\beta) \parallel \mathbb{E}(\frac{1}{p-1}\sum_{j\in C_n,j\neq i}{K}(x^i_{n}-x^j_{n})-\frac{1}{N-1}\sum_{j\neq i}{K}(x^i_{n}-x^j_{n}))\parallel \nonumber \\
    &\le \frac{\beta^n}{N-1}\sum_{j\neq i} \parallel {K}(x^i_{0}-x^j_{0}))-{K}(x^i_{n}-x^j_{n}))\parallel \nonumber \\
    &+\cdots+\frac{\beta(1-\beta)}{N-1}\sum_{j\neq i} \parallel {K}(x^i_{n-1}-x^j_{n-1}))-{K}(x^i_{n}-x^j_{n}))\parallel \nonumber \\
    &+\frac{(1-\beta)}{N-1}\sum_{j\neq i} \parallel {K}(x^i_{n}-x^j_{n}))-{K}(x^i_{n}-x^j_{n}))\parallel \nonumber.
\end{align}
By proposition \ref{proposition1} and the $Lipschitz$ of K and $\parallel x\parallel$ is bounded,we can get it easily:
\begin{align}
    \label{the2-2}\parallel K(x^i_n-x^j_n)-K(x^i_0-x^j_n)\parallel \le Cn\tau.
\end{align}
Substitute the equation (\ref{the2-2}) into the equation (\ref{the2-1}):
\begin{align}\label{the2-3}
\parallel \mathbb{E}(Y_{n,i}(x)) \parallel &\le \frac{\beta^n}{N-1}C_1 n\tau+\frac{\beta^{n-1}(1-\beta)}{N-1}C_2 (n-1)\tau+\cdots+\frac{\beta(1-\beta)}{N-1}C_{n-1} \tau \nonumber\\
&\le \frac{C\tau}{N-1}(n\beta^n+(n-1)\beta^{n-1}+\cdots+\beta).
\end{align}
Using misalignment subtraction:
\begin{align}\label{the2-4}
    n\beta^n+(n-1)\beta^{n-1}+\cdots+\beta=
    \frac{\beta-\beta^{n+1}}{(1-\beta)^2}-\frac{n\beta^{n+1}}{1-\beta}.
\end{align}
Substitute the equation (\ref{the2-4}) into the equation (\ref{the2-3}):
\begin{align}\label{mean_esi}
    \parallel \mathbb{E}(Y_{n,i}(x)) \parallel \le \frac{C\tau}{N-1}
    (\frac{\beta-\beta^{n+1}}{(1-\beta)^2}-\frac{n\beta^{n+1}}{1-\beta}) \le \frac{C\beta \tau}{(1-\beta)^2(N-1)}.
\end{align}
So we can get:
\begin{align*}
    \lim\limits_{\tau \rightarrow 0 \atop \beta \rightarrow 0 } \mathbb{E}(Y_{n,i}(x))=0.
\end{align*}
\end{proof}

Then there are some properties of variance that prove that the variance of   RBM-M is more stable:
\begin{theorem}\label{the_var}
Assume when $t_1 \neq t_2 $, $K(x^i_{t1}-x^j_{t1})$ and $K(x^i_{t2}-x^j_{t2})$ are independent (it is reasonable for the re grouping of different time periods to be independent of each other), and $\tau \to 0, \beta \to 0$, then we have the following estimation formula:

\begin{align}\label{var}
    Var(Y_{n,i}(x))<(\frac{(1-\beta)^2}{1-\beta^2} +o(\beta^{2n+1})) Var(X_{n,i})+2\tau \frac{\beta^2}{(1-\beta^2)^2} \lesssim \frac{(1-\beta)^2}{1-\beta^2}Var(X_{n,i}),
\end{align}
while  $Var(X_{n,i})$ is RBM's variance.
\end{theorem}
\begin{proof}
    According to the definition of variance:
\begin{align*}
    Var(Y_{n,i})=\mathbb{E}(Y_{n,i}^2)-(\mathbb{E}(Y_{n,i}))^2,
\end{align*}
processing them separately:
\begin{align*}
    \mathbb{E}(Y_{n,i}^2)&=\mathbb{E}\mid \tilde{K}^i_n-\frac{1}{N-1} \sum_{i \neq j}K(x^i_{n}-x^j_{n})\mid ^2\\
    &=\mathbb{E}\mid \tilde{K}^i_n \mid ^2-\frac{2}{N-1}\sum_{i \neq j}K(x^i_{n}-x^j_{n}) \mathbb{E}\mid \tilde{K}^i_n \mid+\frac{1}{(N-1)^2}\sum_{i \neq j}K(x^i_{n}-x^j_{n})\sum_{i \neq k}K(x^i_{n}-x^k_{n}),
\end{align*}

\begin{align*}
    (\mathbb{E}(Y_{n,i}))^2&=(\mathbb{E} (\tilde{K}^i_n-\frac{1}{N-1} \sum_{i \neq j}K(x^i_{n}-x^j_{n}))) ^2\\
    &=(\mathbb{E} (\tilde{K}^i_n)) ^2-\frac{2}{N-1}\sum_{i \neq j}K(x^i_{n}-x^j_{n}) \mathbb{E}\mid \tilde{K}^i_n \mid+\frac{1}{(N-1)^2}\sum_{i \neq j}K(x^i_{n}-x^j_{n})\sum_{i \neq k}K(x^i_{n}-x^k_{n}).
\end{align*}
Rewrite the variance based on the two calculation of expectations:
\begin{align*}
    Var(Y_{n,i})&=\mathbb{E}\mid \tilde{K}^i_n \mid ^2-(\mathbb{E} (\tilde{K}^i_n)) ^2\\
    &=\mathbb{E}\mid \frac{\beta^n}{p-1} \sum_{j\in C_0,j\neq i}{K}(x^i_{0}-x^j_{0})+\frac{1-\beta}{p-1}   \sum_{s=1}^{n} \beta^{n-s}  \sum_{j\in C_s,j\neq i}{K}(x^i_{s}-x^j_{s})\mid^2\\
    &-(\mathbb{E}(\frac{\beta^n}{p-1} \sum_{j\in C_0,j\neq i}{K}(x^i_{0}-x^j_{0})+\frac{1-\beta}{p-1}   \sum_{s=1}^{n} \beta^{n-s}  \sum_{j\in C_s,j\neq i}{K}(x^i_{s}-x^j_{s})))^2.
\end{align*}
Use the independence of $K(x^i_{t1}-x^j_{t1})$ and $K(x^i_{t2}-x^j_{t2})$$(t_1 \neq t_2) $:
\begin{align*}
    Var(Y_{n,i})&=\frac{\beta^{2n}}{(p-1)^2}(\mathbb{E}\mid\sum_{j\in C_0,j\neq i}{K}(x^i_{0}-x^j_{0}) \mid^2-(\frac{1}{N-1}\sum_{j\neq i}{K}(x^i_{0}-x^j_{0}))^2)\\
    &+\frac{(1-\beta)^2}{(p-1)^2}   \sum_{s=1}^{n} \beta^{2(n-s)}  \sum_{j\in C_s,j\neq i}(E\mid\sum_{j\in C_0,j\neq i}{K}(x^i_{s}-x^j_{s})\mid^2-(\frac{1}{N-1}\sum_{j\neq i}{K}(x^i_{s}-x^j_{s}))^2).
\end{align*}
Denote \[
X_{s,i}(x) = \frac{1}{p - 1} \sum_{j \in C_i, j \neq i} K(x^{i}_s - x^{j}_s)-\frac{1}{N - 1} \sum_{j \neq i} K(x^{i}_s - x^{j}_s).
\]
Note that $X_{n,i}(x)$ is the variance of RBM, then:
\begin{align*}
    \mid Var(Y_{n,i})\mid&=\mid \beta^{2n}Var(X_{0,i})+\beta^{2(n-1)}(1-\beta)^2 Var(X_{1,i})+\cdots\\
    &+\beta^{2}(1-\beta)^2 Var(X_{n-1,i})+
    (1-\beta)^2 Var(X_{n,i})\mid,
\end{align*}
because $K(\cdot)$ is bound,so:
\begin{align*}
    \mid K(x^{i}_0 - x^{j}_0) - \frac{ 1}{N - 1} \sum_{k_2 \neq i} K(x^{i}_0 - x^{k_2}_0)\mid \le \mid K(x^{i}_n - x^{j}_n) - \frac{1}{N - 1} \sum_{k_1 \neq i} K(x^{i}_n - x^{k_1}_n)\mid+2n\tau,
\end{align*}
hence:
\begin{align*}
    Var(X_{0,i})\le Var(X_{n,i})+2n\tau C,
\end{align*}
add up the variances of all different time steps:
\begin{align*}
    Var(Y_{n,i})&<((1-\beta)^2+\beta^2(1-\beta)^2+\cdots+\beta^{2n-2}(1-\beta)^2+\beta^{2n})Var(X_{n,i})\\
    &+C\tau(2n\beta^{2n}+2(n-1)\beta^{2n-2}+\cdots+\beta^2(1-\beta)^2)\\
    &\le(\frac{(1-\beta)^2}{1-\beta^2}+\frac{2\beta^{2n+1}}{1+\beta})Var(X_{n,i})+\frac{2C\beta^2 \tau}{(1-\beta^2)^2}.
\end{align*}
So when $\tau \to 0$ and ignoring higher-order term of $\beta$,we can get:
\begin{align*}
    Var(Y_{n,i})<\frac{(1-\beta)^2}{1-\beta^2}Var(X_{n,i}).
\end{align*}
\end{proof}

Through the above theoretical analysis, we know that when the hyperparameter $\beta$ is appropriately selected, the error of the   RBM-M algorithm is smaller than that of the RBM, and it exhibits better properties in terms of statistics and probability. Next, we will conduct some numerical experiments to verify our theoretical conclusions from the numerical perspective.

\subsection{The error estimation of rescaled relative entropy}
The joint distribution of all particles simulated by   RBM-M ($\tilde{\rho} _t^N$) can be controlled by numerically solving the corresponding joint distribution (${\rho} _t^N$) for all particles. In order to prove this, the particle association law obtained by solving RBM ($\overline{\rho} _t^N$) is introduced as an auxiliary proof of the intermediate result. What this part needs to estimate is the error about the rescaled relative entropy between   RBM-M and reference solution is:
\begin{align}\label{kl}
    \mathcal{H}_N(\tilde{\rho} _t^N | {\rho}_t^N)= \frac{1}{N} \int_{\mathbb{R}^{N}} \tilde{\rho} _t^N \log\frac{\tilde{\rho}_t^N}{{\rho}_t^N} .
\end{align}

This part of the proof is mainly referenced from \cite{Huang2024MeanFE,Li2022ASU}, using the solution and corrected kernel of   RBM-M to replace the solution and kernel of RBM. 

For convenience, the velocity field $-\nabla V(x)$ is denoted as $b(x)$. For the joint distribution ${\rho} _t^N$ that N particle systems obey at time t, the Fokker-Planck equation is satisfied:
\begin{align} \label{fokker}
    \partial _t {\rho} _t^N+ \sum_{i=1}^{N}div_{x_i}({\rho} _t^N(b(x_i)+K* \rho_t (x_i))) =\sum_{i=1}^{N} \sigma \bigtriangleup _{x_i} {\rho} _t^N.
\end{align}

For the original RBM and   RBM-M solutions, Euler scheme is generally used to solve SDE to simulate the system evolution. 
In order to make the distribution smoother, the solution process is processed continuously and the following equation is obtained:
\begin{align}\label{continuous}
    X_t=X_{T_k}+(t-T_k)b(X_{T_k})+(t-T_k)K^{C_k}(X_{T_k})+\sqrt{2\sigma}(W_{T_{k+1}}-W_{T_k}), \ \ t \in (T_k, T_{k+1}).
\end{align}
In this case, t can continuously take values from $(T_k, T_{k+1})$ for the smooth simulation of each time. \\
For RBM,
\begin{align}
    K^{C_k}(X^i_t) : =\overline{K}_t^i = \frac{1}{p-1} \sum_{j \in C_k, j \ne i} K(X^i_t-X^j_t),
\end{align}
and for   RBM-M,
\begin{align}
    K^{C_k}(X^i_t) : =\tilde{K}_t^i = \frac{1}{p-1} \sum_{j \in C_k, j \ne i} K_{cor}(X^i_t-X^j_t)= \frac{1}{p-1} \sum_{j \in C_k, j \ne i} (\beta K(X^i_t-X^j_t)+(1-\beta) K_{cor}(X^i_{t-1}-X^j_{t-1})).
\end{align}
For continuously solved formula (\ref{continuous}), the continuous solution distributions obtained by RBM and   RBM-M are denoted as $\overline{\varrho } _t^N$ and $\tilde{\varrho }_t^N$. And for them, the Liouville equation holds\cite{Huang2024MeanFE}:
\begin{proposition}
    Denote by $\overline{\varrho}^{N,C}_{t}$ the probability density function of ${X}_{t} = \left( {X}^{1}_{t}, \ldots, {X}^{N}_{t} \right)$ for $t \in [T_{k}, T_{k+1})$. Then the following Liouville equation holds:  
\begin{equation} \label{liou} 
\partial_t \overline{\varrho}^{N,C}_{t} + \sum_{i=1}^{N} \text{div}_{x_{i}} \left( \overline{\varrho}^{N,C}_{t} \left( \overline{b}^{C,i}_{t}(x) + \overline{K}^{C,i}_{t}(x) \right) \right) = \sum_{i=1}^{N} \sigma \Delta_{x_{i}} \overline{\varrho}^{N,C}_{t},  
\end{equation}  
where  
\begin{equation}  
\overline{b}^{C,i}_{t}(x) = \mathbb{E} \left[ b \left( {X}^{i}_{T_{k}} \right) \mid {X}_{t} = x, C \right], \quad t \in [T_{k}, T_{k+1}), 
\end{equation}  
and  
\begin{equation}
\overline{K}^{C,i}_{t}(x) := \mathbb{E} \left[ \overline{K}_t^i \mid {X}_{t} = x, C \right], \quad t \in [T_{k}, T_{k+1}). 
\end{equation}  
Here, $x = (x_{1}, \ldots, x_{n}) \in \mathbb{R}^{Nd}$. 
\end{proposition}

In the same way, using the same proof method, we can get the following lemma of   RBM-M:
\begin{lemma}
    Denote by $\tilde{\varrho}^{N,C}_{t}$ the probability density function of ${X}_{t} = \left( {X}^{1}_{t}, \ldots, {X}^{N}_{t} \right)$ for $t \in [T_{k}, T_{k+1})$. Then the following Liouville equation holds:  
\begin{equation}  
\partial_t \tilde{\varrho}^{N,C}_{t} + \sum_{i=1}^{N} \text{div}_{x_{i}} \left( \tilde{\varrho}^{N,C}_{t} \left( \tilde{b}^{C,i}_{t}(x) + \overline{\tilde{K}^{C,i}_{t}(x)} \right) \right) = \sum_{i=1}^{N} \sigma \Delta_{x_{i}} \tilde{\varrho}^{N,C}_{t},  
\end{equation}  
where  
\begin{equation}  
\tilde{b}^{C,i}_{t}(x) = \mathbb{E} \left[ b \left( {X}^{i}_{T_{k}} \right) \mid {X}_{t} = x, C \right], \quad t \in [T_{k}, T_{k+1}), 
\end{equation}  
and  
\begin{equation}
\overline{\tilde{K}^{C,i}_{t}(x)} := \mathbb{E} \left[ \tilde{K}_t^i \mid {X}_{t} = x, C \right], \quad t \in [T_{k}, T_{k+1}). 
\end{equation}  
\end{lemma}

To prove (\ref{kl}), let's start with Fisher information, which is defined by:
\begin{align}
    \mathcal{I} (\rho)=\int \rho |\nabla\rho|^2 dx,
\end{align}
And, we know that Fisher information statisfies the following inequality:
\begin{align} \label{fisher_inq}
    \frac{1}{\mathcal{I}(p*q)} \ge  \frac{1}{\mathcal{I}(p)}+\frac{1}{\mathcal{I}(q)} ,
\end{align}

Define a definite mapping:
\begin{align} \label{x_map}
    \tilde{\psi}^{C_k}_{\tau}(\mathbf{x}) := \mathbf{x} + \tau \left( \mathbf{b(x)} + \mathbf{K^{C_k}(x)} \right),  
\end{align}
where $\mathbf{x}, \mathbf{b(x)}, \mathbf{K^{C_k}(x)}$ represents the union form of all corresponding variables  respectively,  i.e
\begin{align*}
    &\mathbf{x} = (x_1, \cdots, x_n) \in \mathbb{R}^{Nd},\\
    &\mathbf{b(x)} = (b(x_1), \cdots, b(x_n))^T \in \mathbb{R}^{Nd},  \\
    &\mathbf{\tilde{K}^{C_k}(x)} = (\tilde{K}^{C_k}(x_1), \cdots, \tilde{K}^{C_k}(x_n))^T \in \mathbb{R}^{Nd}.
\end{align*}

Let $\tilde{p}_k(\cdot)$ be the desity of random variable $\tilde{\mathbf{Z}}_k=\tilde{\psi}^{C_k}_{\tau}(\mathbf{X_k})$, then through the conversion of variables, we can get that:
\begin{align*}
    \tilde{p}_k(\mathbf{z})=\frac{\tilde{\varrho}^{N,C}_{T_k}(\mathbf{x})}{det(\nabla \tilde{\psi}^{C_k}_{\tau} \mathbf{x})}.
\end{align*}

Let $q_\tau$ denote the Nd-dimensional Gaussian distribution $\mathcal{N}\left(0,2 \sigma \tau \mathbf{I}_{N d}\right)$. By (\ref{continuous}) and (\ref{x_map}), we have that $\tilde{\varrho}^{N,C}_{T_k}(\mathbf{x})=p_k(\mathbf{x})*q_\tau$. Note that for   RBM-M, the velocity field is the same as for RBM. Moreover, if $K(\cdot)$ is Lipshitz, then $K_{cor}(\cdot)$ is also Lipshitz. Then we have the following estimate for $\mathcal{I}\left(\tilde{\varrho}_{T_{k}}^{N, C}\right)$ upper bound:
\begin{align}
    \mathcal{I}\left(\tilde{\varrho}_{T_{k}}^{N, C}\right) \leq \max \left(\mathcal{I}\left(\rho_{0}^{N}\right), \frac{1+\tau(r+L)}{(1-\tau(r+L))^{2}} M N, \frac{N d(r+L)(3+\tau(r+L))}{2 \sigma}\right) .
\end{align}

For the specific proof process, see \cite{Huang2024MeanFE}.  Recall (\ref{continuous}), the $\tilde{\varrho}_{t}^{N, C}$  is the probability density of the RBM for a given sequence of batches  $\boldsymbol{C}:=\left(C_{0}, C_{1}, \cdots, C_{k}, \cdots\right)$ , so that  $\tilde{\rho}_{T_{k}}^{N}=\mathbb{E}_{\boldsymbol{C}}\left[\tilde{\varrho}_{T_{k}}^{N, \boldsymbol{C}}\right]$ . Moreover, by the Markov property, we are able to define
\begin{align*}
   \tilde{\rho}_{t}^{N, C_{k}}:=\mathbb{E}\left[\tilde{\varrho}_{t}^{N, \boldsymbol{C}} \mid C_{i}, i \geq k\right], \quad t \in\left[T_{k}, T_{k+1}\right).
\end{align*}

Next is a concrete analysis of (\ref{kl}), which is calculated the time derivative.
\begin{theorem}
    Suppose:\\
    (a)the field $b$: $\mathbb{R}^d \to \mathbb{R}^d$ and the interaction kernel $K(\cdot)$ are Lipshitz,\\
    (b)the Hessians of  $b$  and  $K$  have at most polynomial growth:
    \begin{align*}
\left|\nabla^{2} b(x)\right| \leq C(1+|x|)^{q}, \quad\left|\nabla^{2} K(x)\right| \leq \tilde{C}(1+|x|)^{q} .
    \end{align*}
    (c)$\overline{K}^{C,i}-F$  is uniformly bounded, where:
    \begin{align*}
        &F\left(x_{i}\right)=\frac{1}{N} \sum_{j=1}^{N} K\left(x_{i}-x_{j}\right), \overline{K}^{C,i}_{t}(x)=\frac{1}{p} \sum_{j\in C_i} K\left(x_{i}-x_{j}\right), \\
    &\operatorname{esssup}_{C}\left\|\overline{K}^{C,i}_{t}-F\right\|_{L^{\infty}\left(\mathbb{R}^{d}\right)}<+\infty,
    \end{align*}
    then, we have,
    \begin{align*}
    \sup_t\mathcal{H}_{N}\left(\tilde{\rho}_{t}^{N} \mid \rho_{t}^{N}\right) \le \mathcal{H}_{N}\left(\tilde{\rho}_{t}^{0} \mid \rho_{t}^{0}\right)+C_1 \tau^2 +\frac{C_2}{N}.
\end{align*}
\end{theorem}

\begin{proof}
    First we switch the order of integration and derivation, then replace $\partial_t \tilde{\rho}_{t}^{N}$ and $\partial_t \rho_{t}^{N}$ with (\ref{liou}) and (\ref{fokker}), finally we use the Green formula:
\begin{align*}
& \frac{d}{d t} \mathcal{H}_{N}\left(\tilde{\rho}_{t}^{N} \mid \rho_{t}^{N}\right)=\frac{1}{N} \frac{d (\int_{\mathbb{R}^{N}} \tilde{\rho} _t^N \log\frac{\tilde{\rho}_t^N}{{\rho}_t^N})}{dt} \\
=& \frac{1}{N} \int_{\mathbb{R}^{N d}} \partial_t \tilde{\rho}_{t}^{N} log\frac{\tilde{\rho}_{t}^{N}}{\rho_{t}^{N}} d\mathbf{x} +\frac{1}{N} \int_{\mathbb{R}^{N d}} \tilde{\rho}_{t}^{N}(\frac{\partial_t \tilde{\rho}_{t}^{N}}{\tilde{\rho}_{t}^{N}}-\frac{\partial_t \rho_{t}^{N}}{\tilde{\rho}_{t}^{N}}) d\mathbf{x} \\
=&\frac{1}{N} \int_{\mathbb{R}^{N d}}\left(\partial_{t} \tilde{\rho}_{t}^{N}\right)\left(\log \frac{\tilde{\rho}_{t}^{N}}{\rho_{t}^{N}}+1\right) d \mathbf{x}+\frac{1}{N} \int_{\mathbb{R}^{N d}}\left(\partial_{t} \rho_{t}^{N}\right)\left(-\frac{\tilde{\rho}_{t}^{N}}{\rho_{t}^{N}}\right) d \mathbf{x} \\
= & \frac{1}{N} \sum_{i=1}^{N} \int_{\mathbb{R}^{N d}} -\left(\mathbb{E}_{C_{k}} \left[div_{x_i}\left(\tilde{\rho}_{t}^{N, C_{k}}\left(\tilde{b}_{t}^{C_k ,i}(\mathbf{x})+\overline{\tilde{K}_{t}^{C_{k}, i}}(\mathbf{x})\right)\right)-\sigma \bigtriangleup_{x_{i}} \tilde{\rho}_{t}^{N} \right]\right) \left(\log \frac{\tilde{\rho}_{t}^{N}}{\rho_{t}^{N}}+1\right) d \mathbf{x} \\
& +\frac{1}{N} \sum_{i=1}^{N} \int_{\mathbb{R}^{N d}}\left(-div_{x_i} \left(\rho_{t}^{N}\left(b\left(x_{i}\right)+K * \rho_{t}\left(x_{i}\right)\right)\right)+\sigma \bigtriangleup_{x_{i}} \rho_{t}^{N}\right) \left(-\frac{\tilde{\rho}_{t}^{N}}{\rho_{t}^{N}}\right) d \mathbf{x}\\
= & \frac{1}{N} \sum_{i=1}^{N} \int_{\mathbb{R}^{N d}}\left(\mathbb{E}_{C_{k}}\left(\tilde{\rho}_{t}^{N, C_{k}}\left(\tilde{b}_{t}^{C_k ,i}(\mathbf{x})+\overline{\tilde{K}_{t}^{C_{k}, i}}(\mathbf{x})\right)\right)-\sigma \operatorname{div}_{x_{i}} \tilde{\rho}_{t}^{N}\right) \cdot\left(\nabla_{x_{i}} \log \frac{\tilde{\rho}_{t}^{N}}{\rho_{t}^{N}}\right) d \mathbf{x} \\
& +\frac{1}{N} \sum_{i=1}^{N} \int_{\mathbb{R}^{N d}}\left(\rho_{t}^{N}\left(b\left(x_{i}\right)+K * \rho_{t}\left(x_{i}\right)\right)-\sigma \operatorname{div}_{x_{i}} \rho_{t}^{N}\right) \cdot\left(-\nabla_{x_{i}} \frac{\tilde{\rho}_{t}^{N}}{\rho_{t}^{N}}\right) d \mathbf{x} .
\end{align*}

Taking note of $\nabla_{x_i}\log \frac{\tilde{\rho}_{t}^{N}}{\rho_{t}^{N}}=\frac{\rho_{t}^{N}}{\tilde{\rho}_{t}^{N}}\nabla_{x_{i}} \frac{\tilde{\rho}_{t}^{N}}{\rho_{t}^{N}}$ and introducing auxiliary variables $\tilde{K}_t^i$ and $F(x_i)=\frac{1}{N} \sum_{j \ne i}K(x_i-x_j)$ , the above equation is equivalent to
\begin{align}\label{j12345}
&\frac{d}{d t} \mathcal{H}_{N}\left(\tilde{\rho}_{t}^{N} \mid \rho_{t}^{N}\right)=\frac{1}{N} \sum_{i=1}^{N} \int_{\mathbb{R}^{N d}} \mathbb{E}_{C_{k}}\left(\tilde{\rho}_{t}^{N, C_{k}}\left(\tilde{b}_{t}^{C_{k}, i}(\mathbf{x})-b\left(x_{i}\right)\right)\right) \cdot \nabla_{x_{i}} \log \frac{\tilde{\rho}_{t}^{N}}{\rho_{t}^{N}} d \mathbf{x} \\
&+\frac{1}{N} \sum_{i=1}^{N} \int_{\mathbb{R}^{N d}} \mathbb{E}_{C_{k}}\left(\tilde{\rho}_{t}^{N, C_{k}} \overline{\tilde{K}_{t}^{C_{k}, i}}(\mathbf{x})-\tilde{\rho}_{t}^{N, C_{k}} \tilde{K}_t^i\left(x_{i}\right)\right) \cdot \nabla_{x_{i}} \log \frac{\tilde{\rho}_{t}^{N}}{\rho_{t}^{N}} d \mathbf{x} \\
&+\frac{1}{N} \sum_{i=1}^{N} \int_{\mathbb{R}^{N d}} \mathbb{E}_{C_{k}}\left(\tilde{\rho}_{t}^{N, C_{k}} \tilde{K}_t^i\left(x_{i}\right)-\tilde{\rho}_{t}^{N} F\left(x_{i}\right)\right) \cdot \nabla_{x_{i}} \log \frac{\tilde{\rho}_{t}^{N}}{\rho_{t}^{N}} d \mathbf{x} \\
&+\frac{1}{N} \sum_{i=1}^{N} \int_{\mathbb{R}^{N d}}\left(F\left(x_{i}\right)-K * \rho_{t}\left(x_{i}\right)\right) \tilde{\rho}_{t}^{N} \cdot \nabla_{x_{i}} \log \frac{\tilde{\rho}_{t}^{N}}{\rho_{t}^{N}} d \mathbf{x} \\
&-\frac{\sigma}{N} \sum_{i=1}^{N} \int_{\mathbb{R}^{N d}} \tilde{\rho}_{t}^{N}\left|\nabla_{x_{i}} \log \frac{\tilde{\rho}_{t}^{N}}{\rho_{t}^{N}}\right|^{2} d \mathbf{x} .
\end{align}

Among them, unlike \cite{Huang2024MeanFE}, there is no $\mathbb{E}_{C} \left(\tilde{K_t^i}\right)= F(x_i)$, i.e.  $\mathbb{E}_{C} \left(\tilde{K_t^i}\right)$ is no longer an unbiased estimate of $F(x_i)$, so extra processing is required here. Consider another batch $\tilde{C_k}$ independent of $C_k$, so,
\begin{align*}
    &\frac{1}{N} \sum_{i=1}^{N} \int_{\mathbb{R}^{N d}} \mathbb{E}_{C_{k}}\left(\tilde{\rho}_{t}^{N, C_{k}} \tilde{K}_t^i\left(x_{i}\right)-\tilde{\rho}_{t}^{N} F\left(x_{i}\right)\right) \cdot \nabla_{x_{i}} \log \frac{\tilde{\rho}_{t}^{N}}{\rho_{t}^{N}} d \mathbf{x}\\
    =&\frac{1}{N} \sum_{i=1}^{N} \int_{\mathbb{R}^{N d}} \left(\mathbb{E}_{C_{k},\tilde{C_{k}}}\left[\left(\tilde{\rho}_{t}^{N, C_{k}} -\tilde{\rho}_{t}^{N,\tilde{C_{k}}} \right)\left(\tilde{K}_t^i\left(x_{i}\right)-F\left(x_{i}\right)\right)\right]+\mathbb{E}_{C_{k}}\left[\tilde{K}_t^i\left(x_{i}\right)-F\left(x_{i}\right)\right]\tilde{\rho}_{t}^{N}\right) \cdot \nabla_{x_{i}} \log \frac{\tilde{\rho}_{t}^{N}}{\rho_{t}^{N}} d \mathbf{x}.
\end{align*}

Take note of $ab\le ra^2+\frac{1}{4r}b^2$, apply it then,
\begin{align*}
    &\int_{\mathbb{R}^{N d}} \mathbb{E}_{C_{k}, \tilde{C}_{k}}\left[\left(\tilde{K}_t^i\left(x_{i}\right)-F\left(x_{i}\right)\right)\left(\tilde{\rho}_{t}^{N, C_{k}}-\tilde{\rho}_{t}^{N, \tilde{C}_{k}}\right)\right] \cdot \nabla_{x_{i}} \log \frac{\tilde{\rho}_{t}^{N}}{\rho_{t}^{N}} d \mathbf{x} \\
\leq &\frac{2}{\sigma} \mathbb{E}_{C_{k}, \tilde{C}_{k}}\left[\int_{\mathbb{R}^{N d}}\left|\tilde{K}_t^i\left(x_{i}\right)-F\left(x_{i}\right)\right|^{2} \frac{\left|\tilde{\rho}_{t}^{N, C_{k}}-\tilde{\rho}_{t}^{N, \tilde{C}_{k}}\right|^{2}}{\tilde{\rho}_{t}^{N, \tilde{C}_{k}}} d \mathbf{x}\right]+\frac{\sigma}{8} \int_{\mathbb{R}^{N d}} \tilde{\rho}_{t}^{N}\left|\nabla_{x_{i}} \log \frac{\tilde{\rho}_{t}^{N}}{\rho_{t}^{N}}\right|^{2} d \mathbf{x} ,
\end{align*}

\begin{align*}
    &\int_{\mathbb{R}^{N d}} \mathbb{E}_{C_{k}}\left[F\left(x_{i}\right)-\tilde{K}_t^i\left(x_{i}\right)\right]\tilde{\rho}_{t}^{N}\cdot \nabla_{x_{i}} \log \frac{\tilde{\rho}_{t}^{N}}{\rho_{t}^{N}} d \mathbf{x} \\
\leq &\frac{2}{\sigma}\int_{\mathbb{R}^{N d}} \left(\mathbb{E}_{C_{k}}\left[\tilde{K}_t^i\left(x_{i}\right)-F\left(x_{i}\right)\right]\right)^2\tilde{\rho}_{t}^{N} d \mathbf{x}+\frac{\sigma}{8} \int_{\mathbb{R}^{N d}} \tilde{\rho}_{t}^{N}\left|\nabla_{x_{i}} \log \frac{\tilde{\rho}_{t}^{N}}{\rho_{t}^{N}}\right|^{2} d \mathbf{x} .
\end{align*}

Apply (\ref{mean_esi}) of theorem \ref{the_mean}, we have,
\begin{align*}
    \int_{\mathbb{R}^{N d}} \left(\mathbb{E}_{C_{k}}\left[\tilde{K}_t^i\left(x_{i}\right)-F\left(x_{i}\right)\right]\right)^2\tilde{\rho}_{t}^{N} d \mathbf{x}\le c \beta^2 \tau^2.
\end{align*}
where $c$ is a constant. Then, consider the Log-Sobolev inequality,
\begin{align}\label{lsi}
    \mathbb{E}\left(f\log f\right)-\left(\mathbb{E}f\right)\log\left(\mathbb{E}f\right)\le\frac{1}{2}\mathbb{E}\left(f|\nabla \log f|^2\right).
\end{align}
apply (\ref{lsi}) so that,
\begin{align*}
    &\int_{\mathbb{R}^{N d}} \tilde{\rho}_{t}^{N}\left|\nabla_{x_{i}} \log \frac{\tilde{\rho}_{t}^{N}}{\rho_{t}^{N}}\right|^{2} d \mathbf{x}=\mathbb{E}_{\rho^N_t}\left(\frac{\tilde{\rho}_{t}^{N}}{\rho^N_t} \left|\nabla_{x_{i}} \log \frac{\tilde{\rho}_{t}^{N}}{\rho_{t}^{N}}\right|^{2}\right) \\
    \ge &2\mathbb{E}_{\rho^N_t}\left(\frac{\tilde{\rho}_{t}^{N}}{\rho^N_t} \log \frac{\tilde{\rho}_{t}^{N}}{\rho_{t}^{N}}\right)-2\mathbb{E}_{\rho^N_t}\left(\frac{\tilde{\rho}_{t}^{N}}{\rho^N_t}\right) \log \mathbb{E}_{\rho^N_t}\left(\frac{\tilde{\rho}_{t}^{N}}{\rho_{t}^{N}}\right) \ge c \mathcal{H}_{N}\left(\tilde{\rho}_{t}^{N} \mid \rho_{t}^{N}\right).
\end{align*}

The same method is applied to the other items in (\ref{j12345}), and the following results can be obtained by integrating them:
\begin{align*}  
\frac{d}{dt} \mathcal{H}_{N} \left( \tilde{\rho}_t^{N} \big| \rho_t^{N} \right) &\leq -\frac{\sigma}{2C_{LS}} H_N \left( \tilde{\rho}_t^{N} \big| \rho_t^{N} \right) + \frac{2}{\sigma N} \sum_{i=1}^{N} \mathbb{E}_{C_k} \int_{\mathbb{R}^d} \left| \tilde{b}_t^{i} (x) - b(x_i) \right|^2 \tilde{\rho}_t^{N,C_k} dx  \\
&+ \frac{2}{\sigma N} \sum_{i=1}^{N} \mathbb{E}_{C_k} \left[ \int_{\mathbb{R}^d} \left| \overline{\tilde{K}_{t}^{C_{k}, i}}(x_{i}) - \tilde{K}_t^i\left(x_{i}\right) \right|^2 \tilde{\rho}_t^{N,C_k} dx \right]  
\\
&+ \frac{2}{\sigma N} \sum_{i=1}^{N} \mathbb{E}_{C_k,\tilde{C}_k} \left[ \int_{\mathbb{R}^{N d}} \left| \tilde{K}_t^i\left(x_{i}\right) - F(x_i) \right|^2 \frac{\left| \tilde{\rho}_t^{N,C_k} - \tilde{\rho}_t^{N,\tilde{C}_k} \right|^2}{\tilde{\rho}_t^{N, \tilde{C}_k}} dx \right]  
\\
&+ \frac{2}{\sigma N} \sum_{i=1}^{N} \int_{\mathbb{R}^{N d}}  \tilde{\rho}_t^{N} \left| F\left( x_i\right)- K * \rho_t(x_i) \right|^2 dx + c \beta^2 \tau^2\\
&:= -\frac{\sigma}{2C_{LS}} H_N \left( \tilde{\rho}_t^{N} \big| \rho_t^{N} \right)+J_1+J_2+J_3+J_4+c \beta^2 \tau^2,
\end{align*}  
where $C_{LS}$ and $c$ are positive constants. Since RBM and   RBM-M are almost identical except for kernel, the estimation of the same term like $J_1$ and $J_4$ can be directly referred to in the analysis \cite{Huang2024MeanFE}. The estimates of several terms related to $\tilde{K}(\cdot)$ are given here. For $J_2$, we have $\tilde{K}_t^i\left(x_{i}\right)=\mathbb{E}\left[\tilde{K}_t^i\left(\tilde{X}_{t}^{i}\right)\mid \tilde{\mathbf{X}}_{t}=\mathbf{x}, C_{k}\right]$ and apply Tayler's expansion for any $t \in\left[T_{k}, T_{k+1}\right)$:
\begin{align*}
\overline{\tilde{K}_{t}^{C_{k}, i}}(x_i)-\tilde{K}_t^i\left(x_{i}\right) & =\mathbb{E}\left[\tilde{K}_t^i\left(\tilde{X}_{T_{k}}^{i}\right)-\tilde{K}_t^i\left(\tilde{X}_{t}^{i}\right) \mid \tilde{\mathbf{X}}_{t}=\mathbf{x}, C_{k}\right] \\
& =\mathbb{E}\left[\tilde{X}_{T_{k}}^{i}-\tilde{X}_{t}^{i} \mid \tilde{\mathbf{X}}_{t}=\mathbf{x}, C_{k}\right] \cdot \nabla_{x_{i}} \tilde{K}_t^i\left(x_{i}\right)+\hat{r}_{t}\left(x_{i}\right),
\end{align*}

where the remainder takes the form
\begin{align*}
    \hat{r}_{t}\left(x_{i}\right):=\frac{1}{2} \mathbb{E}\left[\left(\tilde{X}_{T_{k}}^{i}-\tilde{X}_{t}^{i}\right)^{\otimes 2}: \int_{0}^{1} \nabla_{x_{i}}^{2} \tilde{K}_t^i\left((1-s) \tilde{X}_{t}^{i}+s \tilde{X}_{T_{k}}^{i}\right) d s \mid \tilde{\mathbf{X}}_{t}=\mathbf{x}, C_{k}\right].
\end{align*}
Note that the Hessians of K have at most polynomial growth through, ignore the higher order of $\beta$ and $ \sup_{t\ge0}\mathbb{E}\left(\left|\tilde{X}_t^i\right|^p|\mathcal{F}_C \right) \le C_p$ through \cite{Jin2020OnTM}:
\begin{align*}
\left|\hat{r}_{t}\left(x_{i}\right)\right| & \leq \int_{0}^{1} \mathbb{E}\left[\left|\nabla_{x_{i}}^{2} \tilde{K}_t^i\left((1-s) \tilde{X}_{t}^{i}+s \tilde{X}_{T_{k}}^{i}\right)\right| \cdot\left|\tilde{X}_{T_{k}}^{i}-\tilde{X}_{t}^{i}\right|^{2} \mid \tilde{\mathbf{X}}_{t}=\mathbf{x}, C_{k}\right] ds\\ 
& \lesssim  C\mathbb{E}\left[\left| \beta \left(1+\left| \tilde{X}_t^i\right|^p+\left| \tilde{X}_{T_k}^i\right|^p \right)+\beta(1-\beta) \left(1+\left| \tilde{X}_{t-1}^i\right|^p+\left| \tilde{X}_{T_k-1}^i\right|^p \right) \right| \cdot\left|\tilde{X}_{T_{k}}^{i}-\tilde{X}_{t}^{i}\right|^{2} \mid \tilde{\mathbf{X}}_{t}=\mathbf{x}, C_{k}\right]\\
& \leq \tilde{C} \mathbb{E}\left[\left|\tilde{X}_{T_{k}}^{i}-\tilde{X}_{t}^{i}\right|^{2} \mid \tilde{\mathbf{X}}_{t}=\mathbf{x}, C_{k}\right] .
\end{align*}

The remaining part of $J_2$ has nothing to do with kernel, and the conclusion can be drawn by referring to the proof of \cite{Huang2024MeanFE}:
\begin{align*}
 J_2 \leq \tilde{c}^{\prime} \tau^{2}\left(3+\frac{1}{N} \mathcal{I}\left(\tilde{\rho}_{T_{k}}^{N}\right)\right).
\end{align*}

Then for the $J_3$, through the assumption that $|\overline{K}_t^i-F(x_i)|$ is bounded and kernel $K(\cdot)$ is Lipschitz continuous:
\begin{align*}
    \left|\tilde{K}_t^i\left(x_{i}\right)-F\left(x_{i}\right)\right| \le \left|\overline{K}_t^i\left(x_{i}\right)-F\left(x_{i}\right)\right|+ \left|\overline{K}_t^i\left(x_{i}\right)-\tilde{K}_t^i\left(x_{i}\right)\right|
    \le C_{bound}+C_{L}L.
\end{align*}
which means $|\tilde{K}_t^i-F(x_i)|$ is also bounded, then we only need to estimate $\mathbb{E}_{C_{k}, \tilde{C}_{k}}\left[\int_{\mathbb{R}^{N d}}\frac{\left|\tilde{\rho}_{t}^{N, C_{k}}-\tilde{\rho}_{t}^{N, \tilde{C}_{k}}\right|^{2}}{\tilde{\rho}_{t}^{N, \varepsilon_{k}}} d \mathbf{x}\right]$ which has been given by \cite{Huang2024MeanFE}. So that we have
\begin{align*}
    J_3 \le \tilde{c}^{\prime\prime} \tau^{2}\left(1+\frac{1}{N} \mathcal{I}\left(\tilde{\rho}_{T_{k}}^{N}\right)\right).
\end{align*}

Finally, combining all of the above results can be obtained
\begin{align*}
    \frac{d}{d t} \mathcal{H}_{N}\left(\tilde{\rho}_{t}^{N} \mid \rho_{t}^{N}\right) \le -C_0 \mathcal{H}_{N}\left(\tilde{\rho}_{t}^{N} \mid \rho_{t}^{N}\right)+C_1 \tau^2 +\frac{C_2}{N},
\end{align*}
then
\begin{align*}
    \sup_t\mathcal{H}_{N}\left(\tilde{\rho}_{t}^{N} \mid \rho_{t}^{N}\right) \le \mathcal{H}_{N}\left(\tilde{\rho}_{t}^{0} \mid \rho_{t}^{0}\right)+C_1 \tau^2 +\frac{C_2}{N}.
\end{align*}
\end{proof}

\section{Numerical experiment}\label{section4}
This section mainly demonstrates a set of numerical experiments to help verify the theories presented in the previous section. In the following numerical experiments, regardless of whether the RBM or   RBM-M algorithm is used, the total number of particles $N$ in the interactive particle system is 10,000, and the number of particles $p$ in each group is 360. This setting takes into account the computational efficiency, and the numerical instability caused by a small number of particles in each group can be avoided. 

The following numerical experiments can correspond to the previous theoretical analysis one by one, including the  convergence order of the kernel function after regularizing, the error of   RBM-M and the good statistical properties of   RBM-M. In addition, the theoretical results are supplemented to some extent by numerical experiments.

\subsection{The error is proportional to $\sqrt{\tau}$}
According to the theorem \ref{rbmm} that we gave the description, the error of   RBM-M and discrete time step $\tau$ should be proportional. 

We select biot-savart to verify this conclusion. In the corresponding interacting particle systems, the corresponding kernel function is\cite{Chaintron2022PropagationOC}:
\begin{align*}
    K(z)=\frac{z^{\bot}}{\parallel z \parallel^2}.
\end{align*}
Here $z^{\bot}$ is the orthogonal complement of $z$. At the same time, for the sake of simplicity, the diffusion term $\sigma$ of the equation is selected to obey the standard normal distribution, and the applied flow $V$ of the system is 0.

The numerical results are recorded in table \ref{table_1}, and the result of numerical experiment that we do can prove it greatly:
\begin{table}[!ht]
    \caption{\textbf{The relationship between the error and the time step.}}
    \label{table_1}
    \centering
    \setlength{\tabcolsep}{20mm}
    {\begin{tabular}{cc}
    \toprule
    $\tau \times 10^{-3}$ &error \\
    \midrule
    16 & 0.0188 \\
    4 & 0.0102 \\
    1 & 0.0051  \\
    0.25 & 0.0025   \\

    \bottomrule
    \end{tabular}}
\end{table}

It can be seen that as the time step of the dispersion becomes $\frac{1}{4}$ of the original, the error of the algorithm basically becomes half of the original. This shows that $\sqrt{\tau}$ is proportional to error, which is exactly what our theory tells us.

\subsection{Compare the error between   RBM-M and RBM }
In this section, we will primarily focus on comparing the algorithmic accuracy of   RBM-M and RBM under various scenarios (primarily involving different kernel functions, representing different interacting particle systems). Some of these interacting particle systems are derived from real physical equations, while others are self-constructed systems with significant singularity, which are specifically designed to verify the stability of the algorithms. The comparison will include both first-order and higher-order systems. By examining several numerical examples, we can comprehensively evaluate the performance between the two algorithms.

As an illustrative example, we will choose a second-order interacting particle system. The specific form of this second-order system is:
\begin{align}\label{position}
     \mathrm{d} X_{t}^{i,N}&=V_{t}^{i,N} dt,
\end{align}
\begin{align}\label{velocity}
      \mathrm{d} V_{t}^{i}=\frac{1}{N} \sum_{j \neq i} K\left(\left|X_{t}^{j}-X_{t}^{i}\right|\right)\left(V_{t}^{j}-V_{t}^{i}\right) \mathrm{d} t+\sigma \mathrm{d} B_{t}^{i},
\end{align}
where the kernel function k has the concrete form:
\begin{align}
    K\left(\left|X_{t}^{j}-X_{t}^{i}\right|\right)=
    \frac{X_{t}^{j}-X_{t}^{i}}{1+\mid X_{t}^{j}-X_{t}^{i} \mid^2} .
\end{align}
In this second-order system, the variable $x$ can be interpreted as the position of the particle, and the variable $v$ can be interpreted as the velocity of the particle. This provides a clear physical interpretation of the entire system. The system is a concrete representation of Newton's second law of motion. Equation (\ref{position}) shows that the position of the particle is determined by the product of velocity and time, while Equation (\ref{velocity}) can be understood as the effect of an external force on the particle, namely the acceleration, and the velocity is determined by the acceleration.

We set the initial conditions of the stochastic differential equation (SDE) such that the initial values of $X$ are distributed uniformly within a unit circle, while the initial values of $V$ are set to 0. Additionally, the evolution time is set to $T = 0.02s$, and the discrete time step is $\tau = 0.001s$. In this case, we will compare the real solution of the interacting particle system, the numerical solution of RBM, and the numerical solution of   RBM-M (with a control parameter $\beta = 0.01$). The results are shown in the figure \ref{figure_2} and \ref{figure_3}.
\begin{figure}[!ht]
        \centering 
        \includegraphics[scale=0.5]{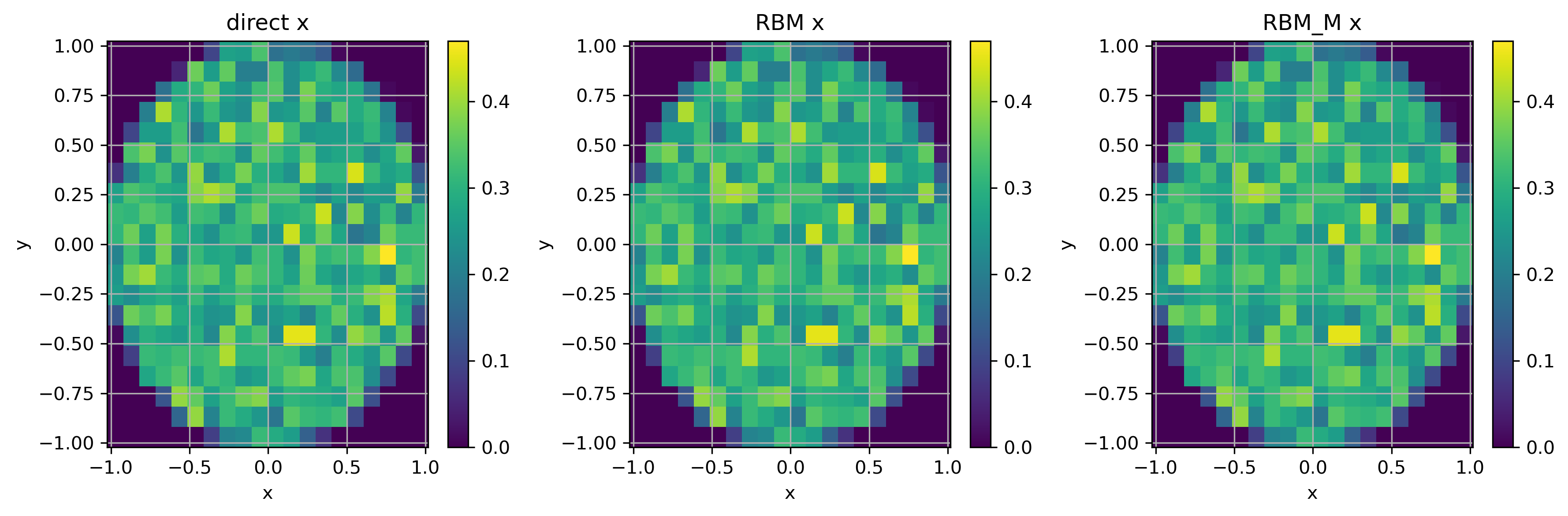}
        \caption{Two order system--x.}
        \label{figure_2}
    \end{figure}

\begin{figure}[!ht]
        \centering 
        \includegraphics[scale=0.5]{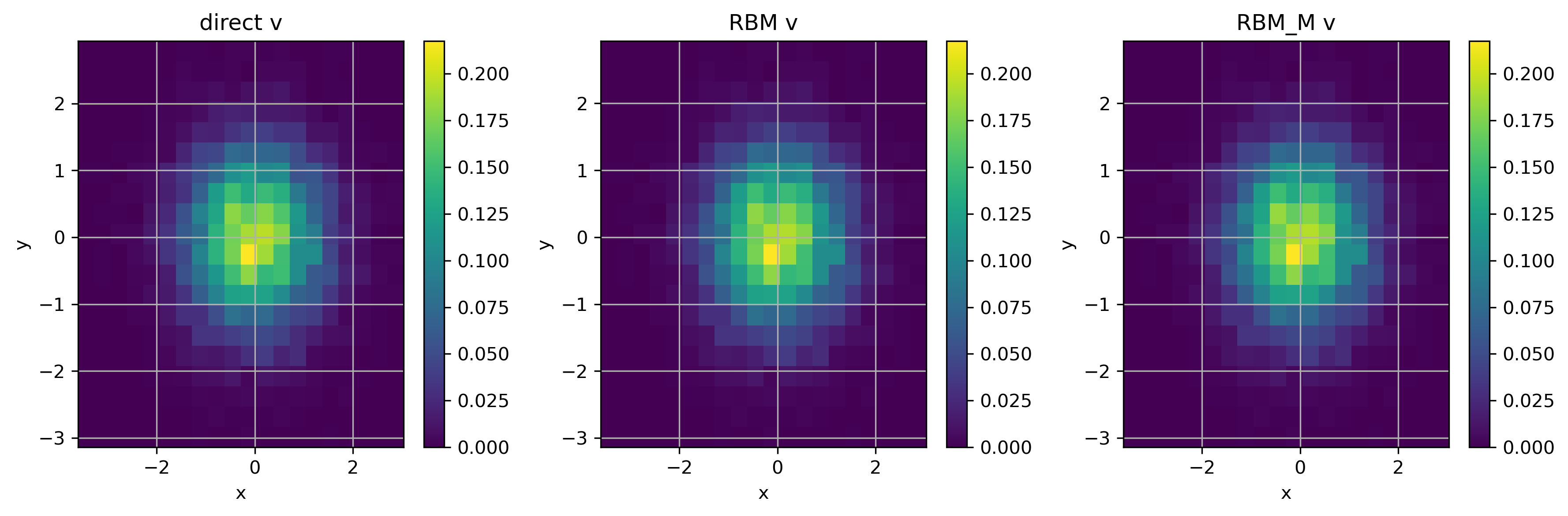}
        \caption{Two order system--v.}
        \label{figure_3}
    \end{figure}
It can be intuitively observed from the above figure that both RBM and   RBM-M have achieved good results. 

In addition to intuitively comparing the results from the figure above, the specific $L_2$ values of RBM and   RBM-M are also listed here. The L2 error here is defined as:
\begin{align}\label{l2}
    l=\sqrt{\sum_{i=1}^{N}(x^i-x^i_{true})^2}.
\end{align}
Here $x^i$ represents the numerical value and $x^i_{true}$ represents the real solution. Under this definition, the specific numerical result is: the error of RBM is $4.489 \times 10^{-2}$, and the error of   RBM-M is $4.419 \times 10^{-2}$. In this example, although the error of   RBM-M is slightly lower than that of RBM, the effect is not significant. We speculate that because this second-order interacting particle system is relatively smooth, the correction brought by the introduction of momentum is not substantial. At the same time, the introduction of momentum may also introduce some errors, so the value of the control parameter $\beta$ is small in this case. However, from another perspective, the performance of   RBM-M should be no worse than RBM, provided that the parameter $\beta$ is reasonably selected.

To further demonstrate that the   RBM-M algorithm can indeed achieve good results when the interacting particle system has a large singularity, we construct the following system:
\begin{align}
    d X^{i}=-\nabla V\left(X^{i}\right) d t+\frac{1}{N-1} \sum_{j \neq i} K\left(X^{i}-X^{j}\right) d t+\sigma d B^{i},
\end{align}
where kernel function is:
\begin{align}
    K_x(s_i-s)=\frac{x_i-x}{\cosh(\parallel s_i-s \parallel^2)},
\end{align}
\begin{align}
    K_y(s_i-s)=\frac{\cosh(y_i-y)}{\parallel s_i-s \parallel^2}.
\end{align}
here $s=(x,y)$ represents the position of the particle, and $K_x(\cdot)$ and $K_y(\cdot)$ represent the components of the kernel function in the X-axis and y-direction, respectively. The additional set system flow is:
\begin{align}
    -\nabla V\left(s_{i}\right)=\binom{0}{cos(x_i)} .
\end{align}
Due to the large singularity of this system,   RBM-M theoretically needs stronger correction effect, so the control parameter $\beta$=0.1 here has a larger value. The remaining settings are consistent with the above second-order system, i.e., the initial value is that the initial values of $X$ are distributed uniformly in a unit circle, the evolution time of termination is $T=0.02s$, and the discrete time step is $\tau=0.001s$. So the standard solution, the solution of RBM and the solution of   RBM-M are respectively in figure \ref{figure_4}.
\begin{figure}[!ht]
        \centering 
        \includegraphics[scale=0.5]{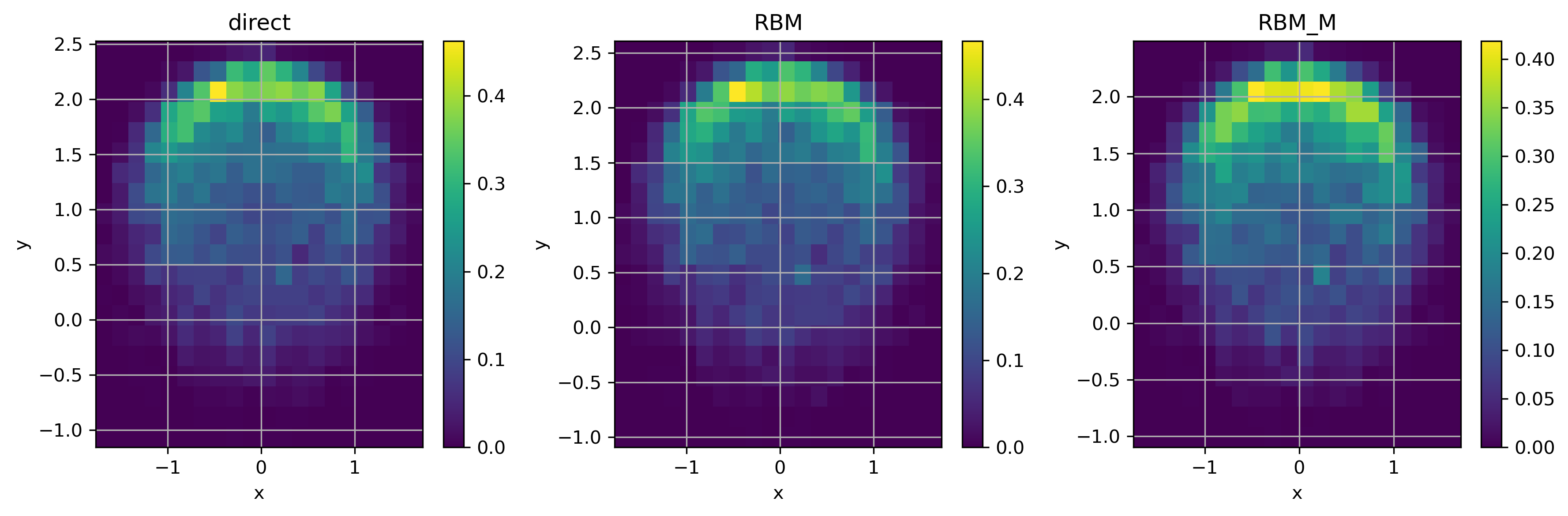}
        \caption{A case of large singularity.}
        \label{figure_4}
    \end{figure}
    
There is still no difference between the two numerical algorithms. Similarly, according to the definition of formula (\ref{l2}), compare the $L_2$ error of the value, where the error of RBM is $7.946\times10^{-2}$, and the error of   RBM-M is only $4.691\times10^{-2}$. In this example, the difference between the two algorithms is still obvious, and the effect of   RBM-M is obviously better than that of RBM.

In addition, it should be noted that in this example, if the original equation is solved, when $N= 10,000$ particles are selected, it takes $54.48s$ to solve, while with the same number of particles, it only takes $4.09s$ to use RBM and $4.45s$ to use   RBM-M. This shows that both RBM and   RBM-M can significantly reduce the computation time, and   RBM-M has little additional time compared to RBM.

At the same time, with this example, the influence of control parameter $\beta$ on the correction effect of   RBM-M algorithm can be explored. We set the value of $\beta$ from $\left \{ 0.04, 0.06, 0.08, 0.1, 0.12 \right \}$ , and take the selected beta as the parameter of   RBM-M to compare the accuracy of   RBM-M algorithm under different parameter conditions, the results are in table \ref{table_2}.

\begin{table}[!ht]
    \caption{\textbf{The effect of parameter $\beta$ on the algorithm.}}
    \label{table_2}
    \centering
    \setlength{\tabcolsep}{10mm}
    {\begin{tabular}{ccc}
    \toprule
    $\beta$ &RBM error  &  RBM-M error\\
    \midrule
    0.04 &$8.029\times10^{-2}$ &$4.529\times10^{-2}$\\
    0.06 & $7.984\times10^{-2}$ &$3.449\times10^{-2}$\\
    0.08 & $8.057\times10^{-2}$  &$3.519\times10^{-2}$\\
    0.10 & $7.946\times10^{-2}$  &$4.695\times10^{-2}$\\
    0.12 & $8.041\times10^{-2}$  & $6.423\times10^{-2}$\\
    \bottomrule
    \end{tabular}}
\end{table}

As can be seen from the above table, with the increase of control parameter $\beta$ from 0.04 to 0.12, the error of   RBM-M changes to the original 56.41\%, \textbf{43.19\%, 43.67\%,} 59.09\% and 79.88\% respectively, which is, when $\beta$ is selected reasonably, the error can be reduced to less than half of the original. Due to the large singularity of the system, the corrective effect of the algorithm is significant. Therefore, in the beginning, as $\beta$ increases, the performance of the   RBM-M algorithm improves gradually. However, as $\beta$ continues to increase, the corrective effect reaches saturation, and the positive impact of further increasing $\beta$ becomes limited, while the introduced error becomes larger and larger, resulting in a deterioration of the final outcome. In summary, when the system singularity is small, the selected $\beta$ should be small; and when the system singularity is large, the selected $\beta$ should be large. In addition, we also tested on other systems, The kernel functions employed in the various systems are as follows: \\
K1: Biot-Savart kernel:
\[K(z)=\frac{z^{\bot}}{\parallel z\parallel}.\]
K2: two order system:

\[ \mathrm{d} X_{t}^{i,N}=V_{t}^{i,N} dt,\]

\[ \mathrm{d} V_{t}^{i}=\frac{1}{N} \sum_{j \neq i} K\left(\left|X_{t}^{j}-X_{t}^{i}\right|\right)\left(V_{t}^{j}-V_{t}^{i}\right) \mathrm{d} t+\sigma \mathrm{d} B_{t}^{i}.\]
K3: repulsive-attractive Morse potential\cite{Carrillo2013ExplicitFS}:

\[ K(\mathbf{x}) = C_r e^{-\frac{\|\mathbf{x}\|}{l_r}} - C_a e^{-\frac{\|\mathbf{x}\|}{l_a}}. \]
K4: Let the particle position be set as $s_i=(x_i, y_i)$

x-dimension:
\[\frac{x_i-x}{\cosh(\parallel s_i-s \parallel^2)},\]

y-dimension:
\[\frac{e^{-(y_i-y)^2}}{\parallel s_i-s \parallel}.\]
K5: Let the particle position be set as $s_i=(x_i, y_i)$

x-dimension:
\[\frac{\sinh(x_i-x)}{\parallel s_i-s \parallel^2},\]

y-dimension:
\[\frac{\cosh(y_i-y)}{\parallel s_i-s \parallel^2}.\]\\
K6: Let the particle position be set as $s_i=(x_i, y_i)$

\[K(z)=\frac{\ln(-exp({z}^2)+1)}{\tan(\parallel z \parallel^2)/10+2\pi}.\]\\
The corresponding $L_2$ error is shown in the table \ref{table_3}.
\begin{table}[H]
    \caption{\textbf{Error contrast between RBM and   RBM-M.}}
    \label{table_3}
    \centering
    \setlength{\tabcolsep}{10mm}
    {\begin{tabular}{ccc}
    \toprule
    kernel &RBM error  &  RBM-M error \\
    \midrule
    1 & $8.0033\times10^{-3}$ &$7.8905\times10^{-3}$\\
    2 & $4.849\times10^{-2}$ &$4.819\times10^{-2}$\\
    3 & $7.817\times10^{-2}$  &$7.792\times10^{-2}$\\
    4 & $7.946\times10^{-2}$  &$4.691\times10^{-2}$\\
    5 & $3.589\times10^{-1}$  & $2.342\times10^{-1}$\\
    6 & $4.289\times10^{-1}$  & $2.423\times10^{-1}$\\

    \bottomrule
    \end{tabular}}
\end{table}

The first column represents the different kinds of kernels, the second column represents the error of RBM solution, and the third column represents the error of    RBM-M solution. It is important to note that due to the different final results evolved from various equations (some distributions are more concentrated, while others are more diffuse), comparing the absolute errors between different equations is not very meaningful. The focus should be on comparing the errors of the two algorithms.

It can be clearly observed that as the kernel singularity increases, the effect of momentum correction becomes more pronounced. This suggests that the local momentum algorithm has significant effects in certain situations.

\subsection{Probabilistic explanation that   RBM-M is more stable than RBM}
Now, let us consider characterizing the Singularity of kernels.
We define a family of similar kernels as:
\begin{align}
     f(x)=\left\{
\begin{aligned}
\ \ -&\frac{1}{x-1} & x<1-\alpha\\
&\frac{1}{\alpha^2} \mid x-1\mid  \ \ \ \ & 1-\alpha<x<1+\alpha\\
&\frac{1}{x-1} \ \ \ \ &x>\alpha+1
\end{aligned}
\right.
,
\end{align}
where $\alpha$ is a parameter to measure the steepness of a function,and $\alpha$ is smaller, the function is more steep ( when selecting a function,  the reason for not taking 0 as the center is to avoid some division of 0 that may lead to a decrease in numerical accuracy). Figure \ref{figure_5} is the figure of $f(x)$:
\begin{figure}[!ht]
        \centering 
        \includegraphics[scale=0.8]{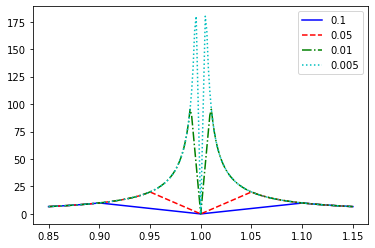}
        \caption{The steepness of different $\alpha$.}
        \label{figure_5}
    \end{figure}

If $\alpha \to 0$, the function $f(x)$ becomes a discontinuous function. In this situation, there is a significant degree of singularity associated with $f(x)$. If we were to use the standard RBM approach in this case, the error would be quite large. However, as demonstrated by the proof presented earlier, the   RBM-M algorithm outperforms the standard RBM when the function exhibits a high degree of singularity. Let us now compare the effects of these two approaches in figure \ref{figure_6}:
\begin{figure}[!ht]
        \centering 
        \includegraphics[scale=0.8]{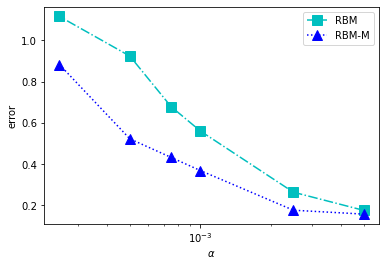}
        \caption{Comparison of errors between RBM and   RBM-M.}
         \label{figure_6}
    \end{figure}

The figure above compares the errors between the two methods. For each data point represented in the figure, the reported error is the average of 10 repeated calculations, in order to mitigate the effects of random factors. The closer the $x$-axis direction is to $1$, the smaller the $\alpha$, which means the function is more steep and has greater singularity. From the graph, it can be observed that as $\alpha$ decreases, the singularity of the function increases, and the error of the standard RBM becomes increasingly larger than that of the   RBM-M method.

The specific data is in table \ref{table_4}.
\begin{table}[!ht]
    \caption{\textbf{Comparison of errors of two algorithms for functions with different singularity.}}
    \label{table_4}
    \centering
    \setlength{\tabcolsep}{10mm}
    {\begin{tabular}{cccc}
    \toprule
    $\alpha$ &RBM  &  RBM-M &advantage\\
    \midrule
    $5\times10^{-3}$ & $0.1743$ &$0.1562$ &$0.0181$\\
    $2.5\times10^{-3}$ & $0.2627$ &$0.1752$ &$0.0875$\\
    $1\times10^{-3}$ & $0.5601$  &$0.3679$ &$0.1923$\\
    $7.5\times10^{-4}$ & $0.6784$  & $0.4317$ &$0.2467$\\
    $5\times10^{-4}$ & $0.9236$  & $0.5218$ &$0.3418$\\
    $2.5\times10^{-4}$ & $1.1161$  & $0.8827$ &$0.2835$\\
       
    \bottomrule
    \end{tabular}}
\end{table}

In this numerical experiment, we set the coefficient of the   RBM-M method, $\beta$, to 0.1. When the function is flatter, we use a smaller $\beta$, and when the function is more steep, we use a larger $\beta$ (the final difference has decreased because the singularity of the function is too large, and a larger $\beta$ should be selected). All of these choices follow the theorem \ref{error}.

Moreover, we can explain the main idea of RBM and   RBM-M from a sampling perspective. (\ref{e1}) and (\ref{e2}) have explained that the core concept is to use the distribution of a few particles to estimate the overall distribution of all particles. When the singularity of the kernel is relatively small, and the distribution is relatively flat, it is easier to estimate. However, if the kernel singularity is relatively high, as mentioned earlier, it becomes more challenging to estimate accurately.
\begin{figure}[!ht]
        \centering 
        \includegraphics[scale=0.8]{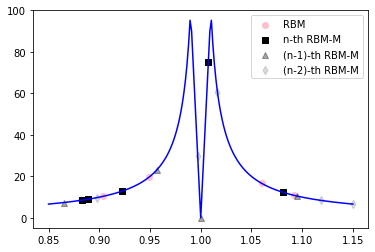}
        \caption{The sampling of RBM and   RBM-M.}
        \label{figure_7}
    \end{figure}
The figure \ref{figure_7} illustrates the sampling processes of two distinct methods. In each time step, the RBM requires a single sampling operation, while the   RBM-M performs multiple sampling iterations due to its consideration of multiple time steps. The pink circular data points represent the particles sampled by the RBM at a specific time step, whereas the black data points depict the samples generated by the   RBM-M. Since each step in the   RBM-M process influences the subsequent evolution of the system, it is analogous to having a greater number of data points being sampled. The transparency of the black data points in the figure indicates the intensity of the influence of the data on the   RBM-M results. The earlier the evolution time, the more transparent the corresponding data points, and the effects of the system's previous evolution will gradually diminish. This visual representation highlights the key differences between the sampling approaches of the RBM and   RBM-M methods, underscoring the importance of considering multiple time steps in the   RBM-M approach and the resulting impact on the data's influence on the overall system evolution.

The above figure illustrates the sampling process of the two methods. In each time step, the standard RBM needs to sample only once, while the   RBM-M samples multiple times by considering multiple time steps. If we directly solve the stochastic differential equations (SDEs), we will definitely sample the steep parts of the function. However, if we use the standard RBM, our sampling may not necessarily represent the entire distribution: if the steep part of the function is not extracted during sampling, the $\mid \frac{1}{p-1} \sum_{j \in C_i, j \neq i} K(x^{i} - x^{j}) \mid$ will be slightly small. Even if the steep part of the function is successfully extracted during sampling, the large outlier will only be weakened by the coefficient $\frac{1}{p-1}$, when it should be accurately using $\frac{1}{N-1}$, so the $\mid \frac{1}{p-1} \sum_{j \in C_i, j \neq i} K(x^{i} - x^{j}) \mid$ will be slightly large. Since the   RBM-M uses exponential average weighting to consider more particles, it can to some extent avoid the errors caused by sampling.

Furthermore, it can be inferred from (\ref{var}) that the variance of   RBM-M is slightly smaller than that of the standard RBM, and both should be (asymptotically) unbiased estimates in an ideal state, so the   RBM-M is more likely to converge to a good result.

\section{Conclusions and future work}\label{section5}
The   RBM-M algorithm is an extension of the RBM algorithm that introduces a momentum correction. Theoretical derivation and numerical experiments have shown that the   RBM-M algorithm can achieve better results than the standard RBM algorithm when the interacting particle system has a relatively large singularity, and the value of the control parameter $\beta$ should be relatively large in such cases. Conversely, when the singularity of the interacting particle system is relatively small, the value of the control parameter $\beta$ should be relatively small.

Furthermore, there is potential to expand the application domain of this algorithm to make it suitable for a wider range of particle systems. For example, other algorithms such as Adam \cite{Kingma2014AdamAM} and RMSProp \cite{XU202117} could be considered for the correction of RBM. Simultaneously, identifying the optimal parameter $\beta$ is also a valuable area of research.

\section*{Data availability}
No data was used for the research described in the article.

\section*{Acknowledgement}
Z.-W. Zhang's research is    .
J.-R. Chen's work    . 
The authors are also very grateful to anonymous reviewers for their valuable comments and suggestions, which have greatly helped to improve our manuscript.

\bibliographystyle{elsarticle-num}
\bibliography{main}
\end{sloppypar}
\end{document}